\documentclass{article}

\usepackage[latin1]{inputenc}
\usepackage{t1enc}
\usepackage{amsmath, amssymb, amsfonts, amsthm, amsopn,epsfig}
\usepackage[none]{hyphenat}
\usepackage{tikz}
\usepackage{float}
\usepackage{graphicx}
\usepackage{caption}
\usepackage[margin=1in]{geometry}
\usepackage[toc,page]{appendix}
\usepackage{epstopdf}
\usepackage{etoolbox}
\usepackage[colorlinks=true]{hyperref}
\usepackage{csquotes}
\usepackage{tikz}
\usepackage{tikz-3dplot}
\usepackage{pgf}
\usetikzlibrary{shapes,calc,positioning}

\theoremstyle{plain}
\newtheorem{theorem}{Theorem}
\newtheorem{definition}{Definition}
\newtheorem{corollary}[theorem]{Corollary}
\newtheorem{claim}[theorem]{Claim}
\newtheorem{lemma}[theorem]{Lemma}

\theoremstyle{definition}

\DeclareMathOperator*{\F}{\mathcal{F}}
\DeclareMathOperator*{\G}{\mathcal{G}}
\DeclareMathOperator{\pr}{pr}

\DeclareMathOperator\ex{ex}
\DeclareMathOperator\La{La}
\DeclareMathOperator\rk{rk}

\author{Istv\'{a}n Tomon\thanks{\'{E}cole Polytechnique F\'{e}d\'{e}rale de Lausanne, Research partially supported by Swiss National Science Foundation grants no. 200020-162884 and 200021-175977.			
		\emph{e-mail}: \textbf{istvan.tomon@epfl.ch}}
}
\title{Forbidden induced subposets of given height}

\begin{document}
\sloppy
\maketitle

\begin{abstract}
Let $P$ be a partially ordered set. The function $\La^{\#}(n,P)$ denotes the size of the largest family $\F\subset 2^{[n]}$ that does not contain an induced copy of $P$. It was proved by  Methuku and P\'{a}lv\"{o}lgyi that there exists a constant $C_{P}$ (depending only on $P$) such that $\La^{\#}(n,P)<C_{P}\binom{n}{\lfloor n/2\rfloor}$. However, the order of the constant $C_{P}$ following from their proof is typically exponential in $|P|$. Here, we show that if the height of the poset is constant, this can be improved. We show that for every positive integer $h$ there exists a constant $c_{h}$ such that if $P$ has height at most $h$, then $$\La^{\#}(n,P)\leq |P|^{c_{h}}\binom{n}{\lfloor n/2\rfloor}.$$ 

Our methods also immediately imply that similar bounds hold in grids as well. That is, we show that if $\F\subset [k]^{n}$ such that $\F$ does not contain an induced copy of $P$ and $n\geq 2|P|$, then $$|\F|\leq |P|^{c_{h}}w,$$
where $w$ is the width of $[k]^{n}$.

A small part of our proof is to partition $2^{[n]}$ (or $[k]^{n}$) into certain fixed dimensional grids of large sides. We show that this special partition can be used to derive bounds in a number of other extremal set theoretical problems and their generalizations in grids, such as the size of families avoiding weak posets, Boolean algebras, or two distinct sets and their union. This might be of independent interest.
\end{abstract}

\section{Introduction}

Let us introduce the basic notation related to partially ordered sets (posets) used throughout the paper, which is mostly conventional. If $P$ is a poset, we denote by $\leq_{P}$ the partial order acting on the elements $P$. If it is clear from the context which poset is under consideration, we may use $\leq$ instead of $\leq_{P}$. A \emph{chain} is a poset, in which any two elements are comparable, and an \emph{antichain} is a poset, in which any two elements are incomparable. The \emph{height} of $P$ is the size of the largest chain in $P$, and the \emph{width} of $P$ is the size of the largest antichain in $P$. The \emph{dimension} of $P$ is the smallest positive integer $d$ for which there exist $d$ linear orderings $\pi_{1},\dots,\pi_{d}:P\rightarrow [|P|]$ such that for $x,y\in P$, we have $x<_{P}y$  if and only if $\pi_{i}(x)<\pi_{i}(y)$ for $i=1,\dots,d$.

 Let $P$ and $Q$ be posets. A \emph{weak copy} of $P$ in $Q$ is a subset $P'$ of $Q$ for which there exists a bijection $\pi:P\rightarrow P'$ such that whenever $x\leq_{P} y$, we have $\pi(x)\leq_{Q} \pi(y)$. Moreover, $P'$ is an \emph{induced copy} of $P$, if $x\leq_{P}y$ holds if and only if $\pi(x)\leq_{Q}\pi(y)$. For example, every chain of size $|P|$ is a weak copy of $P$, but not an induced copy, unless $P$ is also a chain.  We say that a subset $\F\subset Q$ is \emph{(weak $P$)-free}, or \emph{(induced $P$)-free} if $\F$ does not contain a weak copy of $P$, or an induced copy of $P$, respectively.

In this paper, we consider forbidden subposet problems in the Boolean lattice. The \emph{Boolean lattice} $2^{[n]}$ is the power set of $[n]=\{1,\dots,n\}$ ordered by inclusion. For a poset $P$, we define the two functions 
$$\La(n,P)=\max\{|\F|: \F\subset 2^{[n]}\mbox{ is (weak }P\mbox{)-free}\},$$
and
$$\La^{\#}(n,P)=\max\{|\F|: \F\subset 2^{[n]}\mbox{ is (induced }P\mbox{)-free} \}.$$

Forbidden weak and induced subposet problems in the Boolean lattice $2^{[n]}$ are extensively studied. One of the first such results is the classical theorem of Sperner \cite{S}, which states that if $P$ is a chain of size $2$, then $\La(n,P)=\La^{\#}(n,P)=\binom{n}{\lfloor n/2\rfloor}$. This is equivalent to the statement that the width of $2^{[n]}$ is $\binom{n}{\lfloor n/2\rfloor}$. Erd\H{o}s \cite{E} generalized this result in the case $P$ is a chain of size $k$; in this case, $$\La(n,P)=\La^{\#}(n,P)=\max_{l}\sum_{i=l+1}^{l+k-1}\binom{n}{l}\leq (k-1)\binom{n}{\lfloor n/2\rfloor}.$$
Note that this implies that for any poset $P$ on $k$ elements, we have the following bound on the weak subposet problem: $\La(n,P)\leq (k-1)\binom{n}{\lfloor n/2\rfloor}$. The value of $\La(n,P)$ was also studied for a number of fixed posets such as forks and brooms \cite{DK, KT, TH}, diamond \cite{GLL, KMY}, butterfly \cite{DKS}, cycles $C_{4k}$ on two levels \cite{GL}. In case the Hasse diagram of $P$ is a tree, it was proved by Bukh \cite{B} that $\La(n,P)< (h-1+o(1))\binom{n}{\lfloor n/2\rfloor}$, where $h$ is the height of $P$, and Boehnlein and Jiang \cite{BJ} improved this result by showing that $\La^{\#}(n,P)<(h-1+o(1))\binom{n}{\lfloor n/2\rfloor}$ also holds. 

In general, Burcsi and Nagy \cite{BN} and Chen and Li \cite{CL} derived bounds on $\La(n,P)$ depending on the height and size of $P$, which was improved by Gr\'{o}sz, Methuku and Tompkins \cite{GMT} to the asymptotically optimal bound.

\begin{theorem}\label{thm2}(Gr\'{o}sz, Methuku and Tompkins \cite{GMT})
 Let $P$ be a poset of height $h$. Then $$\La(n,P)=O(h\log (|P|/h+2))\binom{n}{\lfloor n/2\rfloor}.$$ 
\end{theorem}

However, it seems that getting general upper bounds in the forbidden induced subposet problem is more challenging. The value of $\La^{\#}(n,P)$ was studied for small posets such as the $2$-fork \cite{CK} and diamond \cite{LM}.  Lu and Milans \cite{LM} proved that if $P$ is a  poset of height $2$, then $\La^{\#}(n,P)=O(|P|\binom{n}{\lfloor n/2\rfloor})$ and they conjectured that for any poset $P$, $\La^{\#}(n,P)/\binom{n}{\lfloor n/2\rfloor}$ is bounded by a constant depending only on $P$. This conjecture was settled by Methuku and P\'{a}lv\"{o}lgyi \cite{MP}.

\begin{theorem}\label{thm1}(Methuku and P\'{a}lv\"{o}lgyi \cite{MP})
	Let $P$ be a poset. There exists a constant $C_{P}$ such that $$\La^{\#}(n,P)\leq C_{P}\binom{n}{\lfloor n/2\rfloor}.$$
\end{theorem}

 However, the constant $C_{P}$ produced by their proof is of the form $2^{d}K_{P}$, where $d$ is the dimension of $P$, and the constant $K_{P}$ comes from a forbidden matrix pattern problem and is also typically exponential in the size of $P$, even for posets of height two. (We shall discuss this in more detail in Section \ref{sect:matrixpatterns} and Section \ref{sect:overview}.) Later, M\'{e}roueh \cite{M} proved that the Lubell mass of any (induced $P$)-free family $\F\subset 2^{[n]}$ is also bounded by a constant $C'_{P}$, but this constant is also expontial in $|P|$.  On the other hand, it is not known whether $C_{P}$ can be chosen to be linear in $|P|$. 

The goal of our paper is to show that if the height of $P$ is a constant, then $\La^{\#}(n,P)/\binom{n}{\lfloor n/2\rfloor}$ is bounded  by a polynomial in $|P|$. Due to the nature of our proof method, our results immediately generalize to yield bounds on the size of (induced $P$)-free families in grids as well, which might be of independent interest. Thus, to be able to state our results in their full generality, we shall first define the notion of grids and cartesian product of posets.

\begin{definition}
If $P_{1}=(X_{1},\leq_{1}),\dots,P_{n}=(X_{n},\leq_{n})$ are partially ordered sets, their \emph{cartesian product}, denoted by $P_{1}\times\dots \times P_{n}$, is the poset $P=(X,\leq)$, where $X=X_{1}\times\dots\times X_{n}$ and for $(x_{1},\dots,x_{n}),(y_{1},\dots,y_{n})\in X$ we have $(x_{1},\dots,x_{n})\leq (y_{1},\dots,y_{n})$ if $x_{i}\leq y_{i}$ holds for $i\in [n]$. 	
\end{definition}

\begin{definition}
Let $k_{1},\dots,k_{n}$ be positive integers larger than $1$. The cartesian product $[k_{1}]\times\dots\times [k_{n}]$ is endowed with a natural point-wise ordering $\preceq$: if $(a_{1},\dots,a_{n}),(b_{1},\dots,b_{n})\in [k_{1}]\times\dots\times [k_{n}]$, then $(a_{1},\dots,a_{n})\preceq (b_{1},\dots,b_{n})$ if $a_{1}\leq b_{1}$,\dots,$a_{n}\leq b_{n}$. We shall refer to this poset structure (and every poset isomorphic to it) as an \emph{$n$-dimensional grid}. The \emph{sides} of the grid $[k_{1}]\times\dots [k_{d}]$ are $k_{1},\dots,k_{d}$. If $k_{1}=\dots=k_{n}=k$, we shall write $[k]^{n}$ instead of $[k_{1}]\times\dots\times [k_{n}]$.	
\end{definition} 

  Note that the Boolean lattice $2^{[n]}$ is isomorphic to $[2]^{n}$, and the cartesian product of $n$ chains is a $n$-dimensional grid. If $C_{1},\dots,C_{n}$ are chains, the function $\pi:C_{1}\times\dots\times C_{n}\rightarrow [|C_{1}|]\times\dots\times [|C_{n}|]$ defined by $\pi((c_{1},\dots,c_{n}))=(l_{1},\dots,l_{n})$, where $c_{i}$ is the $l_{i}$-th smallest element of the chain $C_{i}$ for $i\in [n]$, is called the \emph{natural bijection}. Also, note that the dimension of a poset $P$ is equal to the least positive integer $d$ such that $P$ is an induced subposet of $[|P|]^{d}$.
  
Forbidden subposet problems in grids are less studied. If $P$ is a chain of size $l$, then it is a simple consequence of the so called normalized matching property of $[k]^{n}$ that any $P$-free family $\F\subset [k]^{n}$ has size at most $(l-1)w$, where $w$ is the width of $[k]^{n}$, see \cite{A} for the related definitions and results. The author of this paper \cite{T2} proved the following general result, which extends Theorem \ref{thm1}: for any poset $P$, there exists a constant $C_{P}$ such that if $k,n$ are positive integers, where $n$ is at least the dimension of $P$, then any (induced $P$)-free family $\F\subset [k]^{n}$ satisfies $|\F|\leq C_{P}w$, where $w$ is the width of $[k]^{n}$. 
  
 \subsection{Our results} 
  
  Now let us state the main theorem of this manuscript, which not only implies the previous result for large $n$, but it gives a bound on the constant $C_{P}$ in terms of the size and the height of $P$, which can be viewed as an induced analog of Theorem \ref{thm2}.

\begin{theorem}\label{mainthm}
	For any positive integer $h$ there exists a constant $c_{h}$ such that the following holds. Let $k$ and $n$ be positive integers, $w$ be the width of $[k]^{n}$, and let $P$ be a poset of height $h$ such that $n>2|P|$. If the family $\F\subset [k]^{n}$  does not contain an induced copy of $P$, then $$|\F|\leq |P|^{c_{h}}w.$$
	In particular,
	$$\La^{\#}(n,P)\leq |P|^{c_{h}}\binom{n}{\lfloor n/2\rfloor}$$
    for $n$ sufficiently large.
\end{theorem}

Unfortunately, the dependence of $c_{h}$ on $h$ we are able to prove is quite poor, we get that $c_{h}=2^{O(h\log h)}$. In case $h=2$, our proof yields the following bound, which depends asymmetrically on the two vertex classes of $P$. Here, the two vertex classes of $P$ refer to the set of maximal and minimal elements of $P$.

\begin{theorem}\label{bipartitethm}
	Let $P$ be a poset of height $2$ and let $a\leq b$ be the sizes of the  two vertex classes of $P$. Let $k$ and $n$ be positive integers such that $n>a+\Theta(\log b)$, and let $w$ be the width of $[k]^{n}$. If the family $\F\subset [k]^{n}$  does not contain an induced copy of $P$, then $$|\F|= a^{O(1)}(\log b)^{O(1)}w.$$
	In particular,
	$$\La^{\#}(n,P)= a^{O(1)}(\log b)^{O(1)}\binom{n}{\lfloor n/2\rfloor}$$
    for $n$ sufficiently large.	 
\end{theorem}

As a byproduct of the proof of Theorem \ref{mainthm}, we get a short proof of a bound on the size of (weak $P$)-free families in grids as well, which is only a slightly worse in the case of the Boolean lattice than the one in Theorem \ref{thm2}.

\begin{theorem}\label{weakgrid}
	Let $P$ be a poset of height $h$ and let $k$ and $n$ be positive integers such that $n\geq 2\log_{2}|P|$. Also, let $w$ be the width of $[k]^{n}$. If the family $\F\subset [k]^{n}$ does not contain a weak copy of $P$, then 
	$$|\F|\leq O(wh\log^{3/2} |P|).$$
	In particular,
	$$\La(n,P)\leq O\left(\binom{n}{\lfloor n/2\rfloor}h\log^{3/2}|P|\right).$$
\end{theorem}

Finally, let us remark that in Section \ref{sect:partition}, we present a general technique which might be used to derive bounds for other extremal set theoretical problems as well, and for their generalizations in grids. See Section \ref{sect:applications} for two applications, one involving Boolean algebras, and one on the size of families not containing two distinct sets and their union.

\subsection{Matrix patterns}\label{sect:matrixpatterns}
Let us describe the matrix pattern problem mentioned in the Introduction. The topic discussed in this section is not necessary for our results, but it shares a lot of similarity with the topic of forbidden subposet problems. Also, our proof ideas are partially motivated by the techniques used in this area.

 A \emph{$d$-dimensional matrix pattern} is a $d$-dimensional $0-1$ matrix. If $A$ and $M$ are $d$-dimensional matrix patterns, we say that $M$ \emph{contains} $A$, if we can change some $1$ entries of $M$ to $0$ so that the resulting matrix contains $A$ as a submatrix. We say that $M$ \emph{avoids} $A$ if $M$ does not contain $A$. The \emph{weight} of a matrix $M$ is the number of $1$-s in $M$ and is denoted by $\omega(M)$. A $d$-dimensional \emph{permutation pattern} is a $d$-dimensional matrix pattern $A$ such that every $(d-1)$-dimensional axis-parallel hyperplane of $A$ contains exactly one $1$ entry.  Finally, if $A$ is a $d$-dimensional matrix pattern, let 
$$\ex(n,A)=\max\{\omega(M): M\mbox{ is an }n\times\dots\times n\mbox{ sized }d\mbox{-dimensional matrix pattern avoiding } A\}.$$

The celebrated result of Marcus and Tardos \cite{MT} is that if $A$ is a $k\times k$ sized $2$-dimensional permutation pattern, then there exists a constant $K_{A}$ such that $\ex(n,A)\leq K_{A}n$. It was proved by Fox \cite{Fox} that $K_{A}=2^{O(k)}$, and he proved that $K_{B}=2^{\Omega(k^{1/2})}$ for certain $k\times k$ sized permutation patterns $B$. These results were generalized to higher dimensional permutation patterns by Klazar and Marcus \cite{KM} and Geneson and Tian \cite{GT}. In the latter paper, it is proved that for any $k\times\dots\times k$ sized $d$-dimensional permutation pattern $A$, we have $\ex(n,A)=2^{O_{d}(k)}n^{d-1}$, and there exist $k\times\dots\times k$ sized $d$-dimensional permutation patterns $B$ such that $\ex(n,A)=2^{\Omega_{d}(k^{1/d})}n^{d-1}$. Here, and in the rest of our paper, $O_{p_{1},\dots,p_{s}}(.)$, $\Omega_{p_{1},\dots,p_{s}}(.)$ and $\Theta_{p_{1},\dots,p_{s}}(.)$ mean that the constant hidden in the notation $O(.)$, $\Omega(.)$ and $\Theta(.)$ may depend on the parameters $p_{1},\dots,p_{s}$.

\subsection{Overview of the proof}\label{sect:overview}

To motivate the proof of  Theorem \ref{mainthm}, we shall briefly sketch the proof of Theorem \ref{thm1} by Methuku and P\'{a}lv\"{o}lgyi \cite{MP}. Firstly, they note the following correspondence between forbidden subposet problems and forbidden matrix pattern problems. If $P$ is a $d$-dimensional poset, then there is a $|P|\times\dots\times|P|$ sized $d$-dimensional permutation pattern $A$ such that every (induced $P$)-free family $\F\subset [l]^{d}$ corresponds to an $l\times\dots\times l$ sized $d$-dimensional matrix pattern $M$ avoiding $A$ and satisfying $\omega(M)=|\F|$. But then there exists some constant $K_{P}$ such that $|\F|\leq K_{P}l^{d-1}$. Secondly, using a certain averaging argument in $2^{[n]}$, they prove that having the bound $|\F|\leq K_{P}l^{d-1}$ for (induced $P$)-free families in $[l]^{d}$ implies that $\La^{\#}(n,P)\leq 2^{d}K_{P}\binom{n}{\lfloor n/2\rfloor}$.

Now let us sketch our proof of Theorem \ref{mainthm} in the light of the previous reasoning. First, using a certain partitioning argument, we show that if someone can find two constants $d$ and $K_{P}$ such that for every $l\in \mathbb{Z}^{+}$ and every (induced $P$)-free family $\F'\subset [l]^{d}$, we have $|\F'|\leq K_{P}l^{d-1}$, then  $|\F|\leq O(\sqrt{d}K_{P}w)$ holds for every (induced $P$)-free family $\F\subset[k]^{n}$, where $n\geq d$ and $w$ is the width of $[k]^{n}$. This can be found in Section \ref{sect:partition}. We remark that the same idea was also presented in \cite{T2} by the author of this paper, but in a less general form and with a weaker quantitative bound.

Then, the hearth of the proof is in Section \ref{sect:mainthm}, where we show that we can choose $d\leq |P|+O(\log |P|)$ and $K_{P}=|P|^{c(h)}$, where $c(h)$ is a function depending only on $h$, the height of $P$.

\section{Partitioning into grids}\label{sect:partition}

\subsection{Partitioning lemma}

The following estimate on the width of $[k]^{n}$, which can be found on p.63-68 in \cite{A}, will come in handy later.

\begin{lemma}\label{estimate}
	Let $k,n$ be positive integers such that $k\geq 2$. Then the width of $[k]^{n}$ is $\Theta(k^{n-1}/\sqrt{n}).$	
\end{lemma}

The first idea in the proof of Theorem \ref{mainthm}  is the following heuristic, which can applied in similar problems as well. Instead of trying to bound the size of the maximal (induced $P$)-free family in $[k]^{n}$, where we think of $k$ as fixed and $n$ is tending to infinity (such as the Boolean lattice), we try to find a bound for the same question in $[l]^{d}$, where we think of $d$ as fixed and $l$ is tending to infinity. The following partitioning type lemma and its corollary is the key in connecting these two problems. This lemma was proved in a more general form by the author of this paper in \cite{T}, and in this special form in \cite{T2}. 

\begin{lemma}\label{bigchain}
	Let $k,n$ be positive integers and let $w$ be the width of $[k]^{n}$. Then $[k]^{n}$ can be partitioned into $w$ chains such that the size of each chain is at least $\Omega(k^{n}/w)=\Omega(k\sqrt{n})$.
\end{lemma}

Let us remark that the exact bound appearing in \cite{T} gives $k^{n}/2w-1/2$, but it is more convenient to work with the form $\Omega(k\sqrt{n})$, which is a consequence of Lemma \ref{estimate}. The following immediate corollary of Lemma \ref{bigchain} is what we shall use.

\begin{corollary}\label{biggrid}
	Let $k,n,d$ be positive integers such that $k\geq 2$ and $n\geq d$. Let $m_{1},\dots,m_{d}$ be positive integers such that $m_{1}+\dots+m_{d}=n$ and $m_{1},\dots,m_{d}\in\{\lfloor n/d\rfloor,\lceil n/d\rceil\}$. Then $[k]^{n}$ can be partitioned into $d$-dimensional grids $G_{1},\dots,G_{s}$ such that for each $i\in [s]$, $G_{i}=C_{i,1}\times\dots\times C_{i,d}$, where $C_{i,j}\subset [k]^{m_{j}}$ is a chain of size $\Omega(k\sqrt{n/d})$ for $j\in[d]$.
\end{corollary}

\begin{proof}
	By Lemma \ref{bigchain}, $[k]^{m_{j}}$ can be partitioned into chains $D_{j,1},\dots,D_{j,s_{j}}$ for $j\in [d]$ such that $$|D_{j,l}|=\Omega(k\sqrt{m_{i}})=\Omega(k\sqrt{n/d})$$ for $l\in [s_{j}]$. But then the family of the $d$-dimensional grids $$\{D_{1,l_{1}}\times\dots\times D_{d,l_{d}}\}_{(l_{1},\dots,l_{d})\in [s_{1}]\times\dots\times [s_{d}]}$$
	satisfies the conditions.
\end{proof}

Finally, the following theorem shows the connection between bounds in small dimensional grids and bounds in large dimensional grids for (induced $P$)-free and (weak $P$)-free families.

\begin{theorem}\label{gridtogrid}
	Let $P$ be a poset and $d$ be a positive integer. Suppose that the constant $C$ satisfies that for every positive integer $m$, we have $|\F'|\leq Cm^{d-1}$ for every family $\F'\subset [m]^{d}$ not containing an induced (or weak) copy of $P$. Then for every positive integer $n\geq d$ and $k$, if $\F\subset [k]^{n}$ does not contain an induced (or weak) copy of $P$, then $|\F|=O(C\sqrt{d}w)$, where $w$ is the width of $[k]^{n}$.
\end{theorem}

\begin{proof}
	Let $G_{1},\dots,G_{s}$ be a partition of $[k]^{n}$ into $d$-dimensional grids satisfying the properties of Corollary \ref{biggrid}. Then every side of $G_{i}$ is $\Omega(k\sqrt{n/d})$. Let $\F\subset [k]^{n}$ be a family not containing an induced (or weak) copy of $P$ and let $\F_{i}=\F\cap G_{i}$.
	Let $m$ be the smallest side of $G_{i}$, then by a simple averaging argument, $G_{i}$ contains an $m\times\dots\times m$ sized subgrid $G'_{i}$ such that $|G'_{i}\cap \F_{i}|\geq |\F_{i}|m^{d}/|G_{i}|$. But we have $|G'_{i}\cap \F_{i}|\leq Cm^{d-1}$ as $G'_{i}\cap \F_{i}$ is a subfamily of a grid isomorphic to $[m]^{d}$ and $G'_{i}\cap \F_{i}$ does not contain an induced (or weak) copy of $P$. Thus, we have 
	$$|\mathcal{F}_{i}|\leq \frac{|G'_{i}\cap \mathcal{F}_{i}||G_{i}|}{m^{d}}\leq \frac{C|G_{i}|}{m}=O\left(\frac{C|G_{i}|}{k\sqrt{n/d}}\right).$$ 
	But then  
	$$|\mathcal{F}|=\sum_{i=1}^{s}|\mathcal{F}_{i}|=O\left(\frac{C\sum_{i=1}^{s}|G_{i}|}{k\sqrt{n/d}}\right)=O\left(\frac{Ck^{n}}{k\sqrt{n/d}}\right)=O(C\sqrt{d}w),$$
	where the last equality holds by Lemma \ref{estimate}.
\end{proof}

\subsection{Applications}\label{sect:applications}

Corollary \ref{biggrid} is not only applicable in forbidden subposet problems. In this section, we present two immediate applications of this result, where we bound the size of certain families of $[k]^{n}$ avoiding a fixed substructure.

A subset $B$ of $2^{[n]}$ is a \emph{$d$-dimensional Boolean algebra}, if there exist disjoint sets $X_{0},X_{1},\dots,X_{d}\subset [n]$ such that 
$$B=\left\{X_{0}\cup \bigcup_{i\in I}X_{i}:I\subset [d]\right\}.$$
Let $b(n,d)$ denote the size of a maximal sized family in $2^{[n]}$ not containing a $d$-dimensional Boolean algebra. A $1$-dimensional Boolean algebra is a pair of comparable sets in $2^{[n]}$, so Sperner's theorem \cite{S} gives us that $b(n,1)=\binom{n}{\lfloor n/2\rfloor}=\Theta(2^{n}n^{-1/2})$. In case $d=2$, Erd\H{o}s and Kleitman \cite{EK} proved that $b(n,2)=\Theta(2^{n}n^{-1/4})$. In general, Gunderson, R\"{o}dl and Sidorenko \cite{GRS} showed that $b(n,d)=O_{d}(2^{n}n^{-1/2^{d}})$, where the constant hidden in $O_{d}(.)$ is super exponential in $d$. This was improved by Johnston, Lu and Milans \cite{JLM} to $b(n,d)=O(2^{n}n^{-1/2^{d}})$. Here, we present a short proof of the latter bound $b(n,d)=O(2^{n}n^{-1/2^{d}})$, and show a similar bound for a natural generalization of the problem in grids.

Let us extend the definition of Boolean algebra in grids. Say that two vectors $\mathbf{v},\mathbf{w}\in \{0,1,\dots,k-1\}^{d}$ are \emph{disjoint}, if at most one of $\mathbf{v}(i)$ and $\mathbf{w}(i)$ is nonzero for $i\in [d]$. Then $B\subset [k]^{n}$ is a \emph{Boolean algebra}, if there exist $d+1$ vectors $\mathbf{v}_{0}\in [k]^{n}$ and $\mathbf{v}_{1},\dots,\mathbf{v}_{d}\in \{0,1,\dots,k-1\}^{n}$ such that $\mathbf{v}_{1},\dots,\mathbf{v}_{d}$ are pairwise disjoint and 
$$B=\left\{\mathbf{v}_{0}+\sum_{i\in I}\mathbf{v}_{i}:I\subset [d]\right\}.$$
Note that a Boolean algebra in $[2]^{n}$ corresponds to a Boolean algebra in $2^{[n]}$ in the natural way.

\begin{theorem}
	Let $k,n,d$ be positive integers such that $k\geq 2$ and $n\geq d$. If $\F\subset [k]^{n}$ does not contain a $d$-dimensional Boolean algebra, then 
	$$|\F|=O(k^{n-1/2^{d-1}}n^{-1/2^{d}}).$$
	In particular, 
	$$b(n,d)=O(2^{n}n^{-1/2^{d}}).$$
\end{theorem}

\begin{proof}
	First, let us consider the case $n=d$. In this case, $B\subset [k]^{d}$ is a $d$-dimensional Boolean algebra if and only if there exist $\mathbf{v}_{0},\mathbf{v}_{1}\in [k]^{d}$ such that $\mathbf{v}_{0}(i)<\mathbf{v}_{1}(i)$ for $i\in [d]$ and 
	$$B=\{(\mathbf{v}_{i_{1}}(1),\dots,\mathbf{v}_{i_{d}}(d)):(i_{1},\dots,i_{d})\in\{0,1\}^{d}\}.$$
	But then the problem under consideration is equivalent to an extremal hypergraph Tur\'{a}n problem. Consider the $d$-uniform $d$-partite hypergraph $\mathcal{H}$ with each vertex class being a copy of $[k]$, where $x_{1}\dots x_{d}$ forms a hyperedge if $(x_{1},\dots,x_{d})\in \F$. Then $\F$ does not  contain a $d$-dimensional Boolean algebra if and only if $\mathcal{H}$ does not contain the complete $d$-partite $d$-uniform hypergraph $K_{2,\dots,2}$. But then by a result of Erd\H{o}s \cite{E2},  $\mathcal{H}$ has at most $k^{d-1/2^{d-1}}$ hyperedges, so $|\F|\leq k^{d-1/2^{d-1}}$.

	Now suppose that $n>d$ and let $\F\subset [k]^{n}$ be a family not containing a $d$-dimensional Boolean algebra. From now on, we shall follow a similar line of proof as in Theorem \ref{gridtogrid}. Let $m_{1},\dots,m_{d}$ be positive integers and $G_{1},\dots,G_{s}$ be a partition of $[k]^{n}$ into $d$-dimensional grids defined in the same way as in Corollary \ref{biggrid}. That is, $G_{i}=C_{i,1}\times\dots\times C_{i,d}$ for $i\in [s]$, where $C_{i,j}\in [k]^{m_{j}}$ is a chain of size $\Omega(k\sqrt{n/d})$ for $j\in [d]$.
	
	Again, let $m$ be the smallest side of $G_{i}$, then by a simple averaging argument, $G_{i}$ contains a subgrid $G'_{i}=C'_{i,1}\times\dots C'_{i,d}$ such that $C'_{i,j}\subset C_{i,j}$, $|C'_{i,j}|=m$ and $|G'_{i}\cap \F|\geq |\F_{i}|m^{d}/|G_{i}|$. Consider the natural bijection $\pi: G'_{i}\rightarrow [m]^{d}$. Then it is easy to see that $\pi(\F\cap G'_{i})\subset [m]^{d}$ also does not contain a $d$-dimensional Boolean algebra, so we have $|G'_{i}\cap \F|\leq m^{d-1/2^{d-1}}$. But then 
	$$|\mathcal{G}_{i}\cap \F|\leq \frac{|G'_{i}\cap \mathcal{F}_{i}||G_{i}|}{m^{d}}\leq |G_{i}|m^{-1/2^{d-1}}=O(|G_{i}|(k\sqrt{n/d})^{-1/2^{d-1}}),$$
	which yields
	$$|\mathcal{F}|=\sum_{i=1}^{s}|G_{i}\cap\F|\leq O(\sum_{i=1}^{s}|G_{i}|(k\sqrt{n/d})^{-1/2^{d-1}})=O(k^{n-1/2^{d-1}}n^{-1/2^{d}}),$$ 
	where the last equality holds noting that $1/2<d^{-1/2^{d}}\leq 1$. 
\end{proof}

Let us show another quick application of Corollary \ref{biggrid}. Kleitman \cite{K2} proved that if a family  $\F\subset 2^{[n]}$ does not contain three distinct sets $A,B,C$ such that $A=B\cup C$, then $|\F|\leq \binom{n}{\lceil n/2\rceil}+2^{n}/n$. We shall prove a generalization of a somewhat weaker bound in grids. For $\mathbf{v},\mathbf{w}\in [k]^{n}$, let $$\mathbf{v}\vee \mathbf{w}=(\max\{\mathbf{v}(1),\mathbf{w}(1)\},\dots,\max\{\mathbf{v}(n),\mathbf{w}(n)\}).$$ 

\begin{theorem}
	Let $k,n$ be positive integers larger than $1$. Suppose that the family $\F\subset [k]^{n}$ does not contain three distinct elements $\mathbf{u},\mathbf{v},\mathbf{w}$ such that $\mathbf{u}=\mathbf{v}\vee \mathbf{w}$. Then $|\F|=O(w)$, where $w$ is the width of $[k]^{n}$.
\end{theorem}

\begin{proof}
	First, suppose that $\F\subset [k]\times [l]$ is a family not containing three distinct elements $\mathbf{u},\mathbf{v},\mathbf{w}$ such that $\mathbf{u}=\mathbf{v}\vee \mathbf{w}$. In this case, we show that $|\F|\leq k+l$. To this end, suppose that $|\F|\geq k+l+1$. Say that an element $(a,b)\in \F$ is \emph{$x$-bad}, if there is no $(a',b)\in \F$ with $a'<a$, and say that $(a,b)$ is \emph{$y$-bad}, if there is no $(a,b')\in \F$ such that $b'<b$. Every row of $[k]\times [l]$ contains at most one $x$-bad element, and every column has at most one $y$-bad element, so $\F$ has at most $k+l$ elements that are either $x$-bad or $y$-bad. Hence, $\F$ has an element $\mathbf{u}=(a,b)$ for which there exist $\mathbf{v}=(a',b),\mathbf{w}=(a,b')\in \F$ such that $a'<a$ and $b'<b$. But then $\mathbf{u}=\mathbf{v}\vee \mathbf{w}$, contradiction.
	
	Now let $n\geq 2$ and let $\F\subset [k]^{n}$ such that $\F$ does not contain three distinct elements $\mathbf{u},\mathbf{v},\mathbf{w}$ such that $\mathbf{u}=\mathbf{v}\vee \mathbf{w}$. Apply Corollary \ref{biggrid} to find a partition of $[k]^{n}$ into the $2$-dimensional grids $G_{1},\dots,G_{s}$ such that $G_{i}=C_{i}\times D_{i}$, where $C_{i}\in [k]^{\lfloor n/2\rfloor}$ and $D_{i}\in [k]^{\lceil n/2\rceil}$ are chains of size $\Omega(k\sqrt{n})$. Note that if $\pi: G_{i}\rightarrow [|C_{i}|]\times[|D_{i}|]$ is the natural bijection, then $\pi(\F\cap G_{i})$ also does not contain three distinct elements $\mathbf{u},\mathbf{v},\mathbf{w}$ such that $\mathbf{u}=\mathbf{v}\vee \mathbf{w}$. Hence, we have 
	$$|G_{i}\cap \F|\leq |C_{i}|+|D_{i}|\leq \frac{2|G_{i}|}{\min\{|C_{i}|,|D_{i}|\}}=O\left(\frac{|G_{i}|}{k\sqrt{n}}\right).$$
	But then
	$$|\F|=\sum_{i=1}^{s}|\F\cap G_{i}|\leq O\left(\frac{k^{n}}{k\sqrt{n}}\right)=O(w),$$
	where the last equality holds by Lemma \ref{estimate}.
\end{proof}

\section{Bounds in fixed dimension}\label{sect:mainthm}

The aim of this section is to prove the following theorem.

\begin{theorem}\label{mainthm2}
	There exists a constant $c(h)$ depending only on $h$ such that the following holds. Let $k$ be a positive integer, $d=2|P|$. If the family $\F\subset [k]^{d}$ does not contain an induced copy of $P$, then 
	$$|\F|\leq |P|^{c(h)}k^{d-1}.$$
\end{theorem}

Let us sketch the proof of Theorem \ref{mainthm2} while introducing the notation used in this section. 

\begin{definition}
The \emph{complete multivel poset}, $K_{r_{1},\dots,r_{h}}$, is defined as follows. Let $A_{1},\dots,A_{h}$ be disjoint sets such that $|A_{i}|=r_{i}$ for $i=1,\dots,h$. Let $K=K_{r_{1},\dots,r_{h}}$ be the poset on the set $A_{1}\cup\dots\cup A_{h}$, where $x<_{K}y$ holds if and only if $x\in A_{i}$ and $y\in A_{j}$ for some $1\leq i<j\leq h$. If $r_{1}=\dots=r_{h}=r$, write $K^{h}_{r}$ instead of $K_{r_{1},\dots,r_{h}}$. 	
\end{definition}

Fix a poset $P$ of height $h$ and let $p=|P|$. Then $P$ can be partitioned into $h$ antichains $A_{1},\dots,A_{h}$ such that every element of $A_{i}$ is smaller than or incomparable to any element of $A_{j}$ for $1\leq i<j\leq h$. Take $A_{i}$ to be the set of elements $x\in P$ for which the size of the largest chain starting from $x$ is $i$, for example. We refer to the sets $A_{1},\dots,A_{h}$ as the \emph{levels} of $P$. Let the \emph{rank} of $x\in P$, denoted by $\rk(x)$, be $i$ if $x\in A_{i}$.

Our idea to find an induced copy of $P$ in a family $\F\subset [k]^{d}$ is to first find an induced copy of $K_{|A_{1}|,\dots,|A_{h}|}$ in some lower dimensional slice of $\F$, which we shall correct vertex by vertex to resemble more and more to $P$ by using one new dimension each time. 

To this end, let $q=p-|A_{h}|$ and define the sequence of posets $P_{0},\dots,P_{q}$ as follows. First of all, let $z_{1},\dots,z_{p}$ be an enumeration of the elements of $P$, which satisfies $\rk(z_{1})\leq\dots\leq \rk(z_{p})$. Let $P_{0}$ be the poset on the elements of $P$ such that $x<_{P_{0}}y$ if $\rk(x)<\rk(y)$. Then $P_{0}$ is isomorphic to $K_{|A_{1}|,\dots,|A_{h}|}$. If $P_{l}$ is already defined, where $l<p$, define $P_{l+1}$ as follows. For a poset $Q$, let 
$$\mathcal{R}(Q)=\{(x,y)\in Q^{2}: x<_{Q}y\}$$ be the set of comparable pairs of $Q$. We shall define $P_{l+1}$ by setting $\mathcal{R}(P_{l+1})$. Let
$$S=\{x\in P: z_{l+1}\leq_{P} x\},$$ then set $$\mathcal{R}(P_{l+1})=\mathcal{R}(P_{l})\setminus (S\times(P\setminus S)).$$ 
In other words, we change $P_{l}$ to resemble more to $P$ by correcting the comparabilities involving $z_{l+1}$, and some other comparabilities necessary to keep the transitivity property of $\leq_{P_{l}}$, the resulting poset being $P_{l+1}$.

It is clear that $P_{l+1}$ is a subposet of $P_{l}$, and it can be easily proved by inducton on $l$ that if $1\leq i<j\leq p$ and $i\leq l$, then $z_{i}<_{P_{l}} z_{j}$ if and only if $z_{i}<_{P} z_{j}$. Hence, $P_{q}=P$, as the elements of $A_{h}$ are already forming an antichain in $P_{0}$.

Section \ref{sect:partite} is devoted to finding bounds on the size of (induced $P_{0}$)-free families. Then, in Section \ref{sect:dense}, we prepare the proof of Theorem \ref{mainthm2}, which is presented in Section \ref{sect:sparse}.

\subsection{Chains and complete $h$-partite posets}\label{sect:partite}

\begin{definition}
	Let $Q$ be a poset and let $k,n$ be positive integers. $Q'\subset [k]^{n}$ is a \emph{strong copy} of $Q$ if $Q'$ is an induced copy of $P$, and whenever $\mathbf{x},\mathbf{y}\in Q'$ satisfy $\mathbf{x}\prec\mathbf{y}$, we have strict inequality in every coordinate, that is, $\mathbf{x}(i)<\mathbf{y}(i)$ for $i\in [n]$. Note that we allow incomparable pairs of $Q'$ to have some of their coordinates equal.
\end{definition}

The notion of strong copy is useful for the following reason. Let $\F\subset [k]^{n}$ be a family not containing a strong copy of $P$. Often, when trying to bound the size of $\F$, we shall use induction on $k$ in the following fashion. With some suitable positive integer $s$, we divide $[k]^{n}$ into $s\times\dots\times s$ sized blocks $B_{\mathbf{x}}$ indexed by $\mathbf{x}\in [k/s]^{n}$. If $\F_{1}\subset [k/s]^{n}$ is the family of indices $\mathbf{x}$ such that $\F\cap B_{\mathbf{x}}$ is non-empty, then $\F_{1}$ does not contain a strong copy of $P$ either, so we can use our induction hypothesis to bound the size of $\F_{1}$.

 Note that this is not necessarily true if the notion of strong copy is replaced with induced copy: if $\F=\{(1,2),(3,1)\}\subset [4]^{2}$, then $\F$ does not contain a chain of size $2$; however, if $s=2$, then $\F_{1}=\{(1,1),(2,1)\}\subset [2]^{2}$, which does contain a chain of size $2$. 
 
 Now, our goal is to prove a bound on the size of the maximal family in $[k]^{d}$ not containing a strong copy of $P_{0}\cong K_{|A_{1}|,...,|A_{h}|}$. To this end, we first prove a bound on the size of families not containing a strong copy of a chain of fixed size.

\begin{lemma}\label{strongchain}
	Let $k,d,h$ be positive integers and let $C_{h}$ be a chain of size $h$. If $\F\subset [k]^{d}$ does not contain a strong copy of $C_{h}$, we have 
	$$|\F|\leq d(h-1)k^{d-1}.$$
\end{lemma}

\begin{proof}
	We proceed by induction on $h$. In case $h=1$, the statement is trivial.
	
	Suppose that $h>1$ and let $\F\subset [k]^{d}$ such that $\F$ does not contain a strong copy of $C_{h}$. Let $\F_{d}=\F$ and define the families $\F_{d-1},\dots,\F_{0}$ as follows. Suppose that $\F_{i}$ is already defined. Say that an element $\mathbf{x}\in \F_{i}$ is \emph{bad}, if there is no $\mathbf{y}\in \F_{i}$ such that $\mathbf{x}(j)=\mathbf{y}(j)$ for $[j]\in [d]\setminus\{i\}$ and $\mathbf{x}(i)<\mathbf{y}(i)$. Let  $$\mathcal{F}_{i-1}=\mathcal{F}_{i}\setminus\{\mathbf{x}\in \mathcal{F}_{i}:\mathbf{x}\mbox{ is bad}\}.$$ Each row of $[k]^{d}$ in the direction of the $i$-th coordinate can contain at most one bad element, so there are at most $k^{d-1}$ bad elements in $\F_{i}$. Hence, we have $|\F_{i-1}|\geq |\F_{i}|-k^{d-1}$, which implies $|\F_{0}|\geq |\F|-dk^{d-1}$. 
	
	Also, note that $\F_{0}$ does not contain a strong copy of $C_{h-1}$. To this end, suppose that $C\subset\F_{0}$ is a strong copy of $C_{h-1}$ with maximal element $\mathbf{x}_{0}$. Then, there exist $\mathbf{x}_{1},\dots,\mathbf{x}_{d}$ such that for $i\in [d]$, $\mathbf{x}_{i}\in \F_{i}$, and $\mathbf{x}_{i-1}(j)=\mathbf{x}_{i}(j)$ if $j\in [d]\setminus \{i\}$ and $\mathbf{x}_{i-1}(i)<\mathbf{x}_{i}(i)$. But then $\mathbf{x}_{0}(j)<\mathbf{x}_{d}(j)$ for $j\in [d]$, so $C\cup\{\mathbf{x}_{d}\}$ is a strong copy of $C_{h}$ in $\F$, contradiction.
	
	By induction, we have $|\F_{0}|\leq d(h-2)k^{d-1}$, so $|\F|\leq d(h-1)k^{d-1}$.
\end{proof}

Now we are ready to bound the size of families not containing a strong copy of a complete multilevel poset.

\begin{lemma}\label{strongmultipartite}
	Let $k,d,h,r$ be positive integers such that $2^{d-2}/(d+1)\geq r$. If $\F\subset [k]^{d}$ does not contain a strong copy of $K^{h}_{r}$, then $$|\F|\leq 4d(h-1)k^{d-1}.$$
\end{lemma}

\begin{proof}
	First, we shall prove by induction on $s$, that if $k=2^{s}$, then $|\F|\leq 2d(h-1)k^{d-1}$.
	
	If $s=0$, the statement is trivial. Suppose that $s>0$ and let $\F\subset [k]^{d}$ be a family not containing a strong copy of $K^{h}_{r}$. Divide $[k]^{d}$ into $2\times\dots\times 2$ sized blocks, that is, for $\mathbf{x}\in [k/2]^{d}$, let $B_{\mathbf{x}}=\{2\mathbf{x}-\mathbf{\epsilon}:\mathbf{\epsilon}\in\{0,1\}^{d}\}$. 
	
	Define the families $\F_{1},\F_{2}\subset [k/2]^{d}$ as follows. For $\mathbf{x}\in [k/2]^{d}$, let $\mathbf{x}\in \F_{1}$ if $B_{\mathbf{x}}\cap \F$ is non-empty, and let $\mathbf{x}\in \F_{2}$ if $|B_{\mathbf{x}}|\geq r(d+1)$. Clearly, we have $$|\F|\leq r(d+1)|\mathcal{F}_{1}|+2^{d}|\mathcal{F}_{2}|\leq 2^{d-2}|\mathcal{F}_{1}|+2^{d}|\mathcal{F}_{2}|.$$
	
	Firstly, note that $\F_{1}$ does not contain a strong copy of $K^{h}_{r}$. Indeed, if $S\subset \F_{1}$ is a strong copy of $K^{h}_{r}$, then the set $S'\subset \F$ we get by taking an arbitrary element from each of the blocks $B_{\mathbf{x}}$, $\mathbf{x}\in S$, is also a strong copy of $K^{h}_{r}$ in $\F$, contradiction. Hence, by our induction hypothesis, we have $|\F_{1}|\leq 2d(h-1)(k/2)^{d-1}$.
	
	Secondly, we claim that $\F_{2}$ does not contain a strong copy of a chain of size $h$. Suppose to the contrary, that $S\subset \F_{2} $ is a strong copy of a chain of size $h$. For each $\mathbf{x}\in S$, $B_{\mathbf{x}}\cap \F$ contains at least $r(d+1)$ elements. As the size of the largest chain in $B_{\mathbf{x}}$ is $d+1$, $B_{\mathbf{x}}\cap F$ also contains an antichain $S_{\mathbf{x}}$ of size $r$. But then $\bigcup_{\mathbf{x}\in S}S_{\mathbf{x}}$ is a strong copy of $K^{h}_{r}$, contradiction.
	Hence, Lemma \ref{strongchain} yields that $|\F_{2}|\leq d(h-1)(k/2)^{d-1}.$
	
	We conclude as 
	
	$$|\F|\leq 2^{d-2}|\mathcal{F}_{1}|+2^{d}|\mathcal{F}_{2}|\leq 2^{d-2}2d(h-1)(k/2)^{d-1}+2^{d}d(h-1)(k/2)^{d-1}=2d(h-1)k^{d-1}.$$
	
	The only case remaining is when $k$ is not a power of $2$. Let $2^{s}<k<2^{s+1}$ and let $\F\subset [k]^{n}$ be a family not containing a strong copy of $K^{h}_{r}$. By a simple averaging argument, $[k]^{n}$ contains a $2^{s}\times\dots\times 2^{s}$ sized subgrid $G$ such that $|G\cap \F|\geq (2^{s}/k)^{d}|\F|$. But we have $|G\cap \F|\leq 2d(h-1)2^{s(d-1)}$, which implies $$|\F|\leq 2d(h-1)k^{d-1}(k/2^{s})<4d(h-1)k^{d-1}.$$  
\end{proof}	
	
An immediate corollary of Lemma \ref{strongmultipartite} combined with Theorem \ref{gridtogrid} is Theorem \ref{weakgrid}.

\begin{proof}[Proof of Theorem \ref{weakgrid}]
	As $P$ is a poset of height $h$, it is a weak subposet of $K^{h}_{|P|}$. Setting $d=\lceil 2\log_{2}|P|\rceil+5$, we have $2^{d-2}/(d+1)\geq |P|$. Set $C=4d(h-1)$, then $C=O(h\log |P|)$, and by Lemma \ref{strongmultipartite}, we have $|\F_{0}|\leq Ck^{d-1}$ for every family $\F_{0}\subset [k]^{d}$ not containing an induced copy of $K^{h}_{|P|}$. 
	
	But then Lemma \ref{gridtogrid} implies that $|\F|=O(C\sqrt{d}w)=O(wh\log^{3/2}|P|)$.
\end{proof}

\subsection{Finding $P$ in dense families}\label{sect:dense}

In this section, we prove a  bound of the form $O_{|P|}(k^{d-1/h})$ on the size of (induced $P_{l}$)-free families in $[k]^{d}$. This result is too weak on its own, as we are looking for a bound of the form  $O_{|P|}(k^{d-1})$, but it provides a strong estimate when $k$ is bounded by some polynomial of $|P|$. We shall use this bound in  the proof of Theorem \ref{mainthm2} to handle certain irregular cases, when a family $\F$ has unusually high density in some small parts of $[k]^{d}$.

Let $r=\max_{1\leq i\leq h} |A_{h}|$ and let $d_{0}$ be the smallest positive integer such that $2^{d_{0}-2}/(d_{0}+1)\geq r$.

\begin{lemma}\label{densitylemma}
Let $k$ be a positive integer, $l\in\{0,1,\dots,q\}$ and $d=d_{0}+l$. If the family $\F\subset [k]^{d}$ does not contain a strong copy of $P_{l}$, then 
$$|\F|\leq (8d_{0}(h-1)+4lhk^{(h-1)/h})k^{d-1}.$$
\end{lemma}

\begin{proof}
If $k\leq (8h)^{h}$, the statement is trivial as $$|\F|\leq k^{d}<(8d_{0}(h-1)+4lhk^{(h-1)/h})k^{d-1}.$$ Hence, we can assume that $k\geq 2h$.	

 First, suppose that $k=s^{h}$, where $s$ is a positive integer. The general case will follow using an averaging argument similar to that of in the proof of Lemma \ref{strongmultipartite}.  We shall prove by induction on $l$ that $$|\F|\leq (4d_{0}(h-1)+2lhk^{(h-1)/h})k^{d-1}$$ for any $\F\in [k]^{d}$ not containing a strong copy of $P_{l}$. (Let us remark that $d$ is a function of $l$, which we chose to not denote.) For $l=0$, we have $d=d_{0}$ and $P_{l}=K_{|A_{1}|,\cdots,|A_{h}|}$, which is an induced subposet of $K^{h}_{r}$. Hence, Lemma \ref{strongmultipartite} implies the desired result.

Now suppose that $l\geq 1$ and let $\F\subset [k]^{d}$ be a family not containing a strong copy of $P_{l}$. For $i=0,\dots,h$, define the equivalence relation $\equiv_{i}$ on $[k]$ such that $a\equiv_{i}b$ if $\lceil a/s^{i}\rceil=\lceil b/s^{i}\rceil$. Also, in case $i\geq 1$, let $a\rightarrow_{i} b$, if $a<b$, $a\equiv_{i} b$ but $a\not\equiv_{i-1} b$. Finally, for $\mathbf{x},\mathbf{y}\in [k]^{d}$, let $\mathbf{x}\equiv_{i}\mathbf{y}$ if $\mathbf{x}(j)=\mathbf{y}(j)$ for $j\in[d-1]$ and $\mathbf{x}(d)\equiv_{i} \mathbf{y}(d)$, and let $\mathbf{x}\rightarrow_{i}\mathbf{y}$ if $\mathbf{x}(j)=\mathbf{y}(j)$ for $j\in[d-1]$ and $\mathbf{x}(d)\rightarrow_{i} \mathbf{y}(d)$. 

We shall use the following property of $\rightarrow_{i}$. If $i<j$ and $a\rightarrow_{i} b$, $a\rightarrow_{j} c$, then $b\rightarrow_{j} c$, in particular, $b<c$. Also, if $b\rightarrow_{i} a$ and $c\rightarrow_{j} a$, then $c\rightarrow_{j} b$, in particular, $b<c$.

Say that an element $\mathbf{x}\in \F$ is \emph{good} if for $i=1,\dots,h$, there exist $\mathbf{x}^{(-i)},\mathbf{x}^{(i)}\in \F$ such that $\mathbf{x}^{(-i)}\rightarrow_{i} \mathbf{x}$ and $\mathbf{x}\rightarrow_{i}\mathbf{x}^{(i)}$. We can think of the good elements in the following way as well: divide $[k]^{d}$ perpendicular to the direction of the $d$-th dimension into equal blocks of height $s^{i}$ for $i=1,\dots,h$. Then $\mathbf{x}\in\F$ is good if for $i=1,\dots,h$ there are elements of $\F$ below and above $\mathbf{x}$ that are in the same $s^{i}$ sized block as $\mathbf{x}$, but not in the same $s^{i-1}$ sized block.
 
Let $\mathcal{G}$ be the family of the good elements of $\F$.

\begin{claim}
 $$|\mathcal{G}|\geq|\F|-2hk^{d-1}s^{h-1}.$$
\end{claim}

\begin{proof}
Let $B_{i}$ be the family of elements $\mathbf{x}\in \F$ such that there is no $\mathbf{x}^{(-i)}$ or $\mathbf{x}^{(i)}$ as defined above. 

Note that $[k]^{d}$ has $k^{d-1}s^{h-i}$ pieces of $\equiv_{i}$ equivalence classes, each of size $s^{i}$. Let $E$ be such an equivalence class. Also, $\equiv_{i-1}$ further divides $E$  into $s$ pieces of $\equiv_{i-1}$ equivalence classes $E_{1},\dots,E_{s}$ in increasing order, each of size $s^{i-1}$. However, $B_{i}\cap E$ contains those elements of $\F$ that fall into $E_{a}$ and $E_{b}$, where $a$ is the smallest index such that $\F\cap E_{a}\neq \emptyset$ and $b$ is the largest index such that $\F\cap E_{b}\neq\emptyset$ (if $E\cap \F$ is non-empty). But then $|B_{i}\cap E|\leq |E_{a}|+|E_{b}|\leq 2s^{i-1}$, and as this is true for every $\equiv_{i}$ equivalence class $E$, we get $|B_{i}|\leq 2k^{d-1}s^{h-1}$. Then we are finished as
$$|\mathcal{G}|=|\F\setminus\bigcup_{i=1}^{h}B_{i}|\geq |\F|-2hk^{d-1}s^{d-1}.$$
\end{proof}

By averaging, there exists $t\in [k]$ such that the $(d-1)$-dimensional slice 
$$\mathcal{G}_{t}=\{(x_{1},\dots,x_{d-1}):(x_{1},\dots,x_{d-1},t)\in \mathcal{G}\}\subset[k]^{d-1}$$
contains at least $|\mathcal{G}|/k$ elements. We shall conclude the proof of this lemma by showing that $\mathcal{G}_{t}$ does not contain a strong copy of $P_{l-1}$. Suppose to the contrary that $U=\{\mathbf{u}_{1},\dots,\mathbf{u}_{p}\}\subset \mathcal{G}_{t}$ is a strong copy of $P_{l-1}$, where $\mathbf{u}_{i}$ corresponds to the vertex $z_{i}\in P$. We use these vertices as a reference to find a strong copy of $P_{l}$ in $\F$. 

For $i\in [p]$, let $\mathbf{x}_{i}=(\mathbf{u}_{i}(1),\dots,\mathbf{u}_{i}(d-1),t)\in \G$, and let $X=\{\mathbf{x}_{1},\dots,\mathbf{x}_{p}\}$. Also, define the bijection $\pi: P\rightarrow X$ by setting $\pi(z_{i})=\mathbf{x}_{i}$.

Define the points $\mathbf{y}_{i}\in \F$ as follows. Let $r=\rk(z_{i})$ and $S=\{z\in P: z_{l}\leq_{P_{l}} z\}$. Then

$$\mathbf{y}_{i}=\begin{cases} \mathbf{x}_{i}^{(r)} &\mbox{if } z_{i}\in S,\\
\mathbf{x}_{i}^{(r-1-h)} & \mbox{if } z_{i}\not\in S. \end{cases}$$

See Figure \ref{figure1} for an illustration of $X$ and $Y$.

	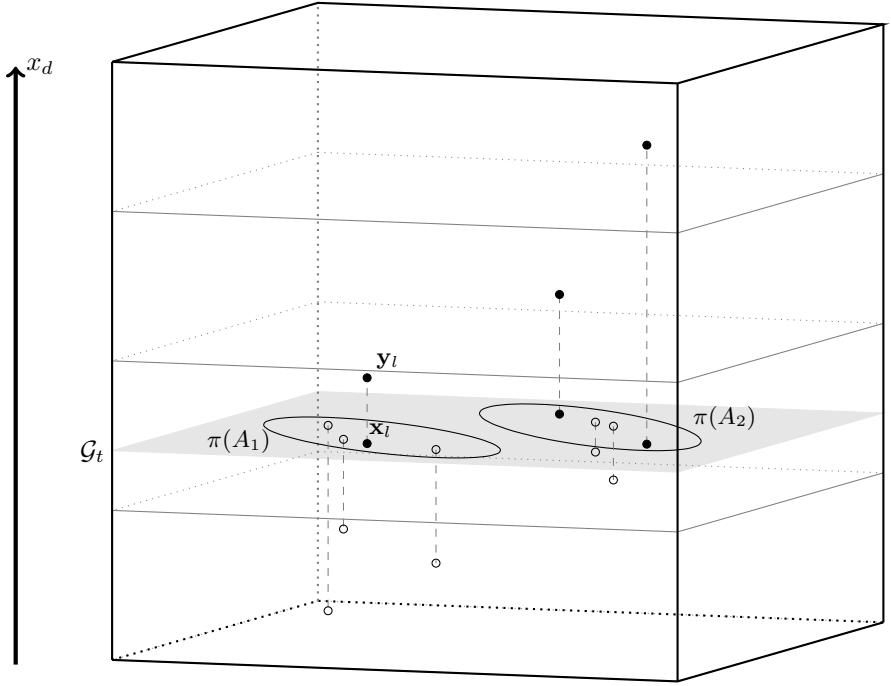
\begin{figure}
		\begin{center}
\tdplotsetmaincoords{84}{20}%84:20
\tdplotsetrotatedcoords{0}{0}{0}%
\begin{tikzpicture}
[tdplot_rotated_coords,
cube/.style={thin,black},
grid/.style={thin,white!90!black},
grid2/.style={thin,gray},
rotated axis/.style={->,black,ultra thick},
projection/.style={dashed,->,blacked, thin},
declare function = {N=8;},
declare function = {h=0.35*N;}
]

%draw a grid in the x-y plane

{
%	\draw[grid] (\x,\y,0) -- (\x,\y,N);
	\filldraw[grid] (0,0,h) -- (0,N,h) -- (N,N,h) -- (N,0,h) -- (0,0,h) ;
%	\draw[grid] (\x,0,\y) -- (\x,N,\y);
}

\foreach \y in {0,N/4,N/2,3*N/4,N}
{
%	\draw[grid2] (\x,\y,0) -- (\x,\y,N);
	\draw[thin, gray] (0,0,\y) -- (N,0,\y);
	\draw[thin, gray] (N,0,\y) -- (N,N,\y);
	\draw[thin, gray,dotted] (N,N,\y) -- (0,N,\y);	
	\draw[thin,gray,dotted] (0,N,\y) -- (0,0,\y);
}

%draw the rotated coordinate frame axes
\draw[rotated axis] (-1,-1,0) -- (-1,-1,N) node[anchor=west]{$x_{d}$};
%\draw[rotated axis] (0,0,0) -- (0,3,0) node[anchor=south west]{$y'$};
%\draw[rotated axis] (0,0,0) -- (0,0,3) node[anchor=west]{$z'$};

%draw the top and bottom of the cube
\draw[thick, black, dotted] (0,0,0) -- (0,N,0) -- (N,N,0) ;
\draw[thick, black] (N,N,0)-- (N,0,0) -- (0,0,0);
\draw[thick, black] (0,0,N) -- (0,N,N) -- (N,N,N) -- (N,0,N) -- (0,0,N);

%draw the edges of the cube
\draw[thick, black] (0,0,0) -- (0,0,N);
\draw[thick, gray, dotted] (0,N,0) -- (0,N,N);
\draw[thick, black] (N,0,0) -- (N,0,N);
\draw[thick, black] (N,N,0) -- (N,N,N);

\coordinate (p1) at (N*0.2,N*0.5,h){} ;
\coordinate (q1) at (N*0.2,N*0.5,h-0.31*N){} ;
\coordinate (p2) at (N*0.5,N*0.2,h){} ;
\coordinate (q2) at (N*0.5,N*0.2,h-0.19*N){} ;
\coordinate (p3) at (N*0.3,N*0.3,h){} ;
\coordinate (q3) at (N*0.3,N*0.3,h-0.15*N){} ;
\coordinate (p4) at (N*0.36,N*0.25,h){} ;
\coordinate (q4) at (N*0.36,N*0.25,h+0.11*N){} ;

\draw[domain=-180:180,smooth,variable=\x,black] plot ({0.35*N+0.25*N*sin(\x)+0.1*N*cos(\x)},{0.35*N+0.1*N*cos(\x)-0.25*N*sin(\x)},h) node[xshift=-25] {\small $\pi(A_{1})$};

\coordinate (p5) at (N*0.8,N*0.4,h){} ;
\coordinate (p6) at (N*0.65,N*0.65,h){} ;
\coordinate (p7) at (N*0.6,N*0.7,h){} ;
\coordinate (p8) at (N*0.5,N*0.8,h){} ;
\coordinate (q5) at (N*0.8,N*0.4,h+0.5*N){} ;
\coordinate (q6) at (N*0.65,N*0.65,h-0.09*N){} ;
\coordinate (q7) at (N*0.6,N*0.7,h-0.05*N){} ;
\coordinate (q8) at (N*0.5,N*0.8,h+0.2*N){} ;

\draw[domain=-180:180,smooth,variable=\x,black] plot ({0.62*N+0.25*N*sin(\x)+0.1*N*cos(\x)},{0.62*N+0.1*N*cos(\x)-0.3*N*sin(\x)},h) node[xshift=80,yshift=5] {\small $\pi(A_{2})$};

\foreach \i in {1,2,...,8}
{
	\draw [dashed, gray] (p\i)--(q\i);
}
\foreach \i in {4,5,8}
{
    \draw[fill=black] (p\i) circle (0.15em) ;
    \draw[fill=black] (q\i) circle (0.15em) ;
}

\foreach \i in {1,2,3,6,7}
{  
	\draw[] (p\i) circle (0.15em) ;
    \draw[] (q\i) circle (0.15em) ;	
}

\node[xshift=5, yshift=5]  at (p4) {\small $\mathbf{x}_{l}$} ;
\node[above right]  at (q4) {\small $\mathbf{y}_{l}$} ;
%\node[right] at (p5) {\tiny $\in \pi(S)$} ;
%\node[right] at (p8) {\tiny $\in \pi(S)$} ;

%\draw [] (p4)--(p5) ;
%\draw [] (p4)--(p8) ;

{\node at (-0.2,-0.2,h) {$\mathcal{G}_{t}$};}
\end{tikzpicture}
\end{center}
\caption{An illustration of how we construct the strong copy of the height $2$ poset $P_{l}$ after finding a strong copy of $P_{l-1}$ in $\mathcal{G}_{t}$. Here, $\bullet$ is used for the elements of $\pi(S)$ and their image in the copy of $P_{l}$, and $\circ$ is used for the elements of $\pi(P\setminus S)$ and their image. } 
\label{figure1}
\end{figure}

\begin{claim}
	$Y=\{\mathbf{y}_{1},\dots,\mathbf{y}_{p}\}$ is a strong copy of $P_{l}$.
\end{claim}

\begin{proof}  First, we show that $Y$ is an induced copy of $P$. Let $t_{1},\dots,t_{p}$ be the $d$-th coordinates of $\mathbf{y}_{1},\dots,\mathbf{y}_{p}$, respectively. Clearly, we have $\mathbf{y}_{i}\preceq y_{j}$ if and only if $\mathbf{x}_{i}\preceq \mathbf{x}_{j}$ and $t_{i}\leq t_{j}$.
	
	Let $\mathcal{R}=\{(z_{i},z_{j})\in P^{2}:\mathbf{y}_{i}< \mathbf{y}_{j}\}$,  $\mathcal{R}_{1}=\{(z_{i},z_{j})\in P^{2}:\mathbf{x}_{i}< \mathbf{x}_{j}\}$ and  $$\mathcal{R}_{2}=\{(z_{i},z_{j})\in P^{2}:i\leq j\mbox{ and }t_{i}< t_{j}\}.$$ We have $\mathcal{R}_{1}=\mathcal{R}(P_{l-1})$ as $X$ is an induced copy of $P_{l-1}$, moreover $\mathcal{R}=\mathcal{R}_{1}\cap \mathcal{R}_{2}$. We want to show that $\mathcal{R}=\mathcal{R}(P_{l})$. Let $z_{i}\in A_{a}$, $z_{j}\in A_{b}$ such that $i<j$ and $a\leq b$. If $a=b$, then $(z_{i},z_{j})\not\in\mathcal{R}_{1}$ as $\mathcal{R}_{1}=\mathcal{R}(P_{l-1})\subset \mathcal{R}(P_{0})$. Now suppose that $a<b$. If $z_{i}\in S$, we have $t\rightarrow_{a} t_{i}$, and if $z_{i}\not\in S$, $t_{i}\rightarrow_{h+1-a} t$. Similarly, if $z_{j}\in S$ , we have $t\rightarrow_{b} t_{j}$, otherwise, $t_{j}\rightarrow_{h+1-b} t$. Hence, $t_{i}\not<t_{j}$ if and only if $z_{i}\in S$ and $z_{j}\in P\setminus S$. But then 
	$$\mathcal{R}=\mathcal{R}_{1}\cap\mathcal{R}_{2}=\mathcal{R}_{1}\setminus (S\times (P\setminus S))=\mathcal{R}(P_{l}),$$ so $Y$ is truly an induced copy of $P$.
	
	But $Y$ is also a strong copy of $P$. To this end, let $1\leq i<j\leq p$ such that $\mathbf{y}_{i}< \mathbf{y}_{j}$. Then $\rk(z_{i})<\rk(z_{j})$. Also, the first $d-1$ coordinates of $\mathbf{y}_{i}$ and $\mathbf{y}_{j}$ form $\mathbf{u}_{i}$ and $\mathbf{u}_{j}$, respectively, and $U$ is a strong copy of $P_{l-1}$, so we have strict inequality in the first $d-1$ coordinates. Moreover, we have $t_{i}<t_{j}$ as  $t_{i}=t_{j}$ can only occur if $\rk(z_{i})=\rk(z_{j})$.
	
\end{proof}

We showed that $\mathcal{G}_{t}$ does not contain a strong copy of $P_{l-1}$, so our induction hypothesis yields that 
$$|\mathcal{G}_{t}|\leq (4d_{0}(h-1)+2(l-1)hs^{h-1})k^{d-2}.$$

But we also have 
$$|\mathcal{G}_{t}|\geq |\mathcal{G}|/k\geq (|\F|-2hk^{d-1}s^{h-1})/k,$$
which implies
$$|\F|\leq (4d_{0}(h-1)+2lhs^{h-1})k^{d-1}.$$

This finishes the proof of the case when $\sqrt[h]{k}$ is an integer. Now suppose that $s^{h}<k<(s+1)^{h}$, where $s$ is a positive integer. As $k> (8h)^{h}$, we have $s\geq 4h$, so $k/s^{h}<(s+1)^{h}/s^{h}<(1+1/4h)^{h}<2$. Let $\F\subset [k]^{d}$ be a family not containing a strong copy of $P_{l}$. By a simple averaging argument, $[k]^{d}$ contains a $s^{h}\times\dots\times s^{h}$ sized subgrid $G$ such that $|G\cap \F|\geq (s^{h}/k)^{d}|\F|$. But we have $|G\cap \F|\leq (4d_{0}(h-1)+2lhs^{h-1})s^{h(d-1)}$, which implies $$|\F|\leq (4d_{0}(h-1)+2lhs^{h-1})s^{h(d-1)}(k/s^{h})^{d}<(8d_{0}(h-1)+4lhk^{(h-1)/h})k^{d-1}.$$
\end{proof}

Let us remark that Lemma \ref{densitylemma} shows a real difference between the forbidden matrix pattern problem and the forbidden induced subposet problem. Fox \cite{Fox} showed the existence of an $l\times l$ sized $2$-dimensional permutation pattern $A$ such that for $k=2^{\Omega(l^{1/4})}$, we have
$$\frac{\ex(k,A)}{k^{2}}>1-O(l^{-1/2}).$$
 In other words, there are exponentially sized matrices of unusually high relative weight avoiding $A$. In contrast, Lemma \ref{densitylemma} tell us that families in $[k]^{d}$ of positive density must contain an induced copy of $P$, if $k=|P|^{\Omega_{h}(1)}$.

\subsection{Finding $P$ in sparse families}\label{sect:sparse}

Now Theorem \ref{mainthm2} is an immediate consequence of the following lemma, noting that $P_{q}=P$ and $d_{0}+q+1\leq 2|P|$. Here, $d_{0}$ is the constant defined in the beginning of Section \ref{sect:dense}.

\begin{lemma}\label{mainlemma}
There exists a constant $c(h)$ depending only on $h$ such that the following holds. Let $k$ be a positive integer, $l\in\{0,1,\dots,q\}$ and $d=d_{0}+l+1$. If the family $\F\subset [k]^{d}$ does not contain a strong copy of $P_{l}$, then 
$$|\F|\leq |P|^{c(h)}k^{d-1}.$$
\end{lemma}

The proof of this lemma shares similar elements with the proof of Marcus and Tardos \cite{MT} on matrix patterns, but it is more involved. Being familiar with their proof might be helpful, but it is not necessary. Before we can embark on our proof, we need the following technical lemma.

\begin{lemma}\label{intersectionlemma}
	Let $0<\alpha<1/2$, let $V$ be a finite set  and let $V_{1},\dots,V_{m}\subset V$ such that $|V_{i}|\geq \alpha|V|$. If $m\geq 2h/\alpha$, then there exists $I\subset [m]$, $|I|=h$ such that $$\left|\bigcap_{i\in I}V_{i}\right|\geq (\alpha/12)^{h+1}|V|.$$
\end{lemma}

\begin{proof}
	Let $M=\lceil 2h/\alpha \rceil$. Let $W$ be the set of elements $v\in V$ that are contained in at least $h$ different sets among $V_{1},\dots,V_{M}$. Then 
	$$h|V|+M|W|\geq \sum_{i}^{M} |V_{i}|\geq \alpha M|V|,$$
	which gives $|W|\geq (\alpha-h/M)|V|\geq\alpha|V|/2$. For each element of $v\in W$, choose $h$ sets among $V_{1},\dots,V_{M}$ containig $v$. As there are at most $\binom{M}{h}\leq (eM/h)^{h}\leq (12/\alpha)^{h}$ choices for these $h$ sets, there are $h$ different sets $V_{i_{1}},\dots,V_{i_{h}}$, whose intersection has size at least $$|W|(\alpha/12)^{h}\geq |V|(\alpha/12)^{h+1}.$$
\end{proof}

\begin{proof}[Proof of Lemma \ref{mainlemma}]
To overcome certain divisibility conditions and to avoid the use of floors and ceilings, we first consider the case when $k$ is a power of $2$, and also set certain other parameters to be powers of $2$. Then, the inequality for general $k$ follows from the same averaging argument as in Lemma \ref{strongmultipartite}, where we lose only a factor of $2$, which is absorbed by the term $|P|^{c(h)}$. Define the sequence $s_{0},\dots,s_{h}$ as follows. Let $s_{0}=1$, and if $s_{i}$ is already defined for $i\in \{0,\dots,h-2\}$, set $s_{i+1}$ to be the smallest power of $2$ larger than $$(100p^{3}hs_{i})^{2h^2}.$$ Finally, set $s_{h}=k$. Note that for $i\in [h-1]$, $s_{i}=p^{2^{O(i\log h)}}$. In particular, $s_{h-1}=p^{2^{O(h\log h)}}$.

Also, define the constants $C_{0},\dots C_{q}$ as follows. Let $C_{0}=2h^{2}p^{3}s_{h-1}^{2}$ and if $C_{l}$ is already defined, let $C_{l+1}=(1+8h/p)C_{l}$. Note that $C_{q}<(1+8h/p)^{p}C_{0}<e^{8h}C_{0}$, so there exists a function $c(h)=2^{O(h\log h)}$ such that $C_{q}<p^{c(h)}$.

We shall prove by induction on $l$ that if $\F\subset [k]^{d}$ does not contain a strong copy of $P_{l}$, then ${|\F|\leq C_{l}k^{d-1}}$. Let us draw the attention of our reader to the detail that $d$ is a function of $l$ both here and in the rest of the proof, which we choose not to denote for easier readability. Also, this $d$ is exactly $1$ larger than the corresponding $d$ defined in Lemma \ref{densitylemma}.

 If $l=0$, we have $P_{l}=K_{|A_{1}|,\dots,|A_{h}|}$ for which  Claim \ref{strongmultipartite} implies $|\F|\leq 4d(h-1)k^{d}< C_{0}k^{d-1}$. Now suppose that $l\geq 1$. We shall proceed by induction on $k_{0}$, where $k=2^{k_{0}}$. If $k\leq C_{l}$, then the statement is trivial as $|\F|\leq k^{d}\leq C_{l}k^{d-1}.$ So suppose that $k>C_{l}$, then $k>s_{h-1}$ holds as well.

For a subset $A\subset [k]^d$, let $\pr(A)$ denote the $(d-1)$-dimensional projection of $A$ along the $d$-th coordinate axis, that is, $$\pr(A)=\{(x_{1},\dots,x_{d-1}): \exists x_{d}\mbox{ s.t. }(x_{1},\dots,x_{d})\in A\}.$$

 For $i=1,\dots,h-1$, divide $[k]^{d}$ into $s_{i}\times\dots\times s_{i}$ sized blocks $B^{i}_{\mathbf{u}}$, where $\mathbf{u}\in [k/s_{i}]^{d}$ and $$B^{i}_{\mathbf{u}}=\{s_{i}\mathbf{u}-\mathbf{v}:\mathbf{v}\in \{0,\dots,s_{i}-1\}^{d}\}.$$ 
 As we choose $k$ and $s_{i}$ to be powers of $2$, and $k\geq s_{h-1}$, $k/s_{i}$ is an integer. Let us label the elements $\mathbf{u}\in[k/s_{i}]^{d}$ (and therefore the blocks) according to the behaviour of $B^{i}_{\mathbf{u}}\cap\F$. Say that 
 
 $$\mathbf{u}\mbox{ is }\begin{cases}\mbox{empty,} &\mbox{ if }  B^{i}_{\mathbf{u}}\cap \F=\emptyset;\\
 									 \mbox{light,} &\mbox{ if } 1\leq |B^{i}_{\mathbf{u}}\cap \F|\leq s_{i}^{d-1}/p;\\
 									 \mbox{fat,}   &\mbox{ if } |\pr (B^{i}_{\mathbf{u}}\cap\F)|\geq s_{i}^{d-1}/p^{2}s_{i-1};\\
 									 \mbox{medium,}&\mbox{  otherwise.} \end{cases}$$

Let $\mathcal{I}^{i}_{1}$ be the family of light elements of $[k/s_{i}]^{d}$, $\mathcal{I}^{i}_{2}$ be the family of fat elements, and $\mathcal{I}^{i}_{3}$ be the family of medium elements. Also, for $j\in [3]$, let $\F^{i}_{j}=\bigcup_{\mathbf{u}\in\mathcal{I}^{i}_{j}} (B^{i}_{\mathbf{u}}\cap\F$). 

\begin{claim}\label{lightblocks}
$|\mathcal{I}^{i}_{1}|\leq C_{l}(k/s_{i})^{d-1}.$
\end{claim}

\begin{proof}
 Note that $\mathcal{I}^{i}_{1}$ does not contain a strong copy of $P_{l}$. Otherwise, if $\mathbf{x}_{1},\dots,\mathbf{x}_{p}\in \mathcal{I}^{i}_{1}$ are the vertices of a strong copy of $P_{l}$, then for any choice $\mathbf{y}_{1}\in B^{i}_{\mathbf{x}_{1}},\dots,\mathbf{y}_{p}\in B^{i}_{\mathbf{x}_{p}}$, the vertices $\mathbf{y}_{1},\dots,\mathbf{y}_{p}$ form a strong copy of $P_{l}$ in $\F$. Hence, by our induction hypothesis, we have $|\mathcal{I}^{i}_{1}|\leq C_{l}(k/s_{i})^{d-1}$. % and $$\sum_{\mathbf{u}\in \F_{1}}|B_{\mathbf{u}}^{i}\cap\F|\leq f(k/s_{i})s_{i}^{d-1}/p.$$
\end{proof}

Now let us bound the number of elements of $\F$ contained in a fat box.

\begin{claim}\label{fatblocks}
$|\mathcal{I}^{i}_{2}|\leq 2hp^{2}s_{i-1}(k/s_{i})^{d-1}.$
\end{claim}

\begin{proof}
Let $\alpha=1/p^{2}s_{i-1}$ and $m=2hp^{2}s_{i-1}=2h/\alpha$. It is enough to show that there exists no $(x_{1},\dots,x_{d-1})\in [k/s_{i}]^{d-1}$ such that the row $$\{(x_{1},\dots,x_{d-1},t):t\in [k/s_{i}]\}$$ contains more than $m$ fat elements. 

Suppose to the contrary that there exist $\mathbf{w}_{1},\dots,\mathbf{w}_{m}\in \mathcal{I}^{i}_{2}$ having the same first $d-1$ coordinates. Let $V=\pr(B^{i}_{\mathbf{w}_{1}})$ and let $V_{j}=\pr(B^{i}_{\mathbf{w}_{j}}\cap\F).$ Then $V_{j}\subset V$, and $|V_{j}|\geq \alpha |V|$ by the definition of fat.

But then Lemma \ref{intersectionlemma} implies that there exist $h$ sets among $V_{1},\dots,V_{m}$, whose common intersection has size at least $(\alpha/12)^{h+1}|V|$. Without loss of generality, suppose that $V_{1},\dots,V_{h}$ are such sets, the $d$-th coordinates of $\mathbf{w}_{1},\dots,\mathbf{w}_{h}$ are in increasing order, and let $\mathcal{G}=\bigcap_{j=1}^{h} V_{j}$. (See Figure \ref{figure2}.)

But $\mathcal{G}$ is a subfamily of $V\cong[s_{i}]^{d-1}$, where $d=d_{0}+l$, so Lemma \ref{densitylemma} tells us that $\mathcal{G}$ contains a strong copy of $P_{l}$ provided 
\begin{equation*}
|\mathcal{G}|>(8d_{0}(h-1)+4lhs_{i}^{(h-1)/h})s_{i}^{d-2}.
\end{equation*}
But this inequlaity holds, as we have $|\mathcal{G}|\geq (1/12p^{2}s_{i-1})^{h+1}s_{i}^{d-1}$ and $s_{i}\geq(100p^{3}hs_{i-1})^{2h^{2}}$. (Let us omit the simple but messy calculations.)

Let $S$ be a strong copy of $P_{l}$ in $\mathcal{G}$ and let $\pi: P_{l}\rightarrow S$ be an isomorphism. Define the injection $\pi': P_{l}\rightarrow\F$ as follows: if $z\in P_{l}$ such that $r=\rk(z)$, then let $\pi'(z)$ be an arbitrary element $\mathbf{x}\in B^{i}_{\mathbf{w_{r}}}\cap \F$ for which $\pr(\mathbf{x})=\pi(z)$. Then $\pi'(P_{l})$ is a strong copy of $P_{l}$ in $\F$, contradiction.
\end{proof}

\pgfmathsetseed{6}
	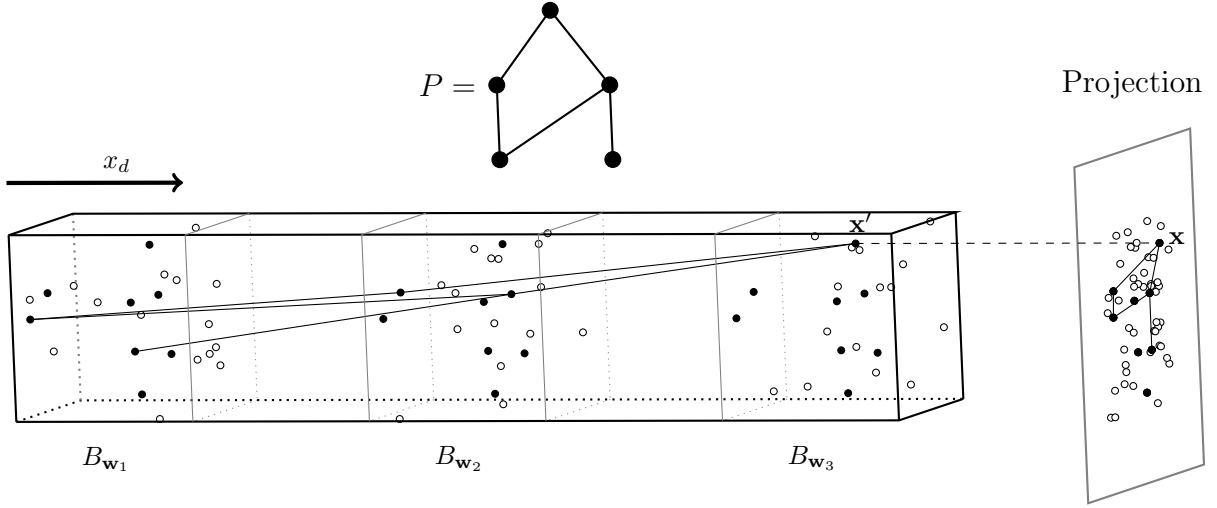
\begin{figure}
	\begin{center}
		\tdplotsetmaincoords{83}{20}%84:20
		\tdplotsetrotatedcoords{0}{-2.5}{0}%
		\begin{tikzpicture}
		[tdplot_rotated_coords,
		cube/.style={thin,black},
		grid/.style={thin,white!90!black},
		grid2/.style={thin,gray},
		rotated axis/.style={->,black,ultra thick},
		projection/.style={dashed,->,blacked, thin},
		declare function = {N=2.5;},
		declare function = {h=0.35*N;},
		declare function = {numBox=5;},
		declare function = {l=N*numBox;}
		]
		
		\coordinate (a1) at (7,0,N+1){} ;
		\coordinate (a2) at (8.6,0,N+1){} ;
		\coordinate (a3) at (7,0,N+2){} ;
		\coordinate (a4) at (8.6,0,N+2){} ;
		\coordinate (a5) at (7.8,0,N+3){} ;
		
		\foreach \i in {1,...,5}
		{
			\draw[fill=black] (a\i) circle (0.3em) ;
		}
		\draw[thick, black] (a1) -- (a3);
		\draw[thick, black] (a1) -- (a4);
		\draw[thick, black] (a2) -- (a4);
		\draw[thick, black] (a3) -- (a5);
		\draw[thick, black] (a4) -- (a5);
		\node[] at (6.3,0,N+2) {\large  $P=$} ;

		\foreach \y in {1,2,3,4}
		{
			
			\draw[thin, gray,dotted] (\y*N,N,N) -- (\y*N,N,0);	
			\draw[thin,gray,dotted] (\y*N,N,0) -- (\y*N,0,0);
		}

		%draw the rotated coordinate frame axes
		\draw[rotated axis] (0,0,N+0.7) -- (N,0,N+0.7) node[anchor=south,xshift=-25]{$x_{d}$};
		
		%draw the top and bottom of the cube
		\draw[thick, black, dotted] (0,0,0) -- (0,N,0) -- (l,N,0) ;

		%draw the edges of the cube
	
		\draw[thick, gray, dotted] (0,N,0) -- (0,N,N);
		
		\node[] at (N/2,0,-0.5) {$B_{\mathbf{w}_{1}}$};
		\node[] at (2*N+N/2,0,-0.5) {$B_{\mathbf{w}_{2}}$};
		\node[] at (4*N+N/2,0,-0.5) {$B_{\mathbf{w}_{3}}$};
		
		\draw[thick, gray] (l+3,-1,-1) -- (l+3,-1,N+1) -- (l+3,N+1,N+1) -- (l+3,N+1,-1) -- cycle;
		\node[] at (l+3.5,0,N+2) {\large Projection};

		\foreach \i in {1,...,15}
		{
			\foreach \j in {0,2,4}
			{
				\pgfmathsetmacro\x{N*\j+N/2+N*rand/2} ;
				\pgfmathsetmacro\y{N/2+N*rand/2} ;
				\pgfmathsetmacro\z{N/2+N*rand/2} ;
		  
		    	\coordinate (p\j\i) at (\x,\y,\z){} ;
		    	\draw[fill=white] (p\j\i) circle (0.13em) ;
		    	\draw[fill=white] (l+3,\y,\z) circle (0.13em) ;
		    }
		}

		\foreach \i in {1,...,8}
		{
				\pgfmathsetmacro\x{N/2+N*rand/2} ;
				\pgfmathsetmacro\y{N/2+N*rand/2} ;
				\pgfmathsetmacro\z{N/2+N*rand/2} ;
				
				\foreach \j in {0,2,4}
				{
					%	\coordinate (p\j\i) at (N*\j+N/2+N*rand/2,N/2+N*rand/2,N/2+N*rand/2){} ;
					\coordinate (q\j\i) at (\x+\j*N,\y,\z){} ;
					\draw[fill=black] (q\j\i) circle (0.13em) ;
					
					\coordinate (r\i) at (l+3,\y,\z) ;
					\draw[fill=black] (r\i) circle (0.13em)  ;
				
				}
		}
		
		\draw[thin, black] (r5) -- (r3) ;
		\draw[thin, black] (r5) -- (r4) ;
		\draw[thin, black] (r7) -- (r4) ;
		\draw[thin, black] (r3) -- (r6) ;
		\draw[thin, black] (r4) -- (r6) ;
		\node[right,yshift=2] at (r6) {$\mathbf{x}$} ;
		
		\draw[thin, black] (q05) -- (q23) ;
		\draw[thin, black] (q05) -- (q24) ;
		\draw[thin, black] (q07) -- (q24) ;
		\draw[thin, black] (q23) -- (q46) ;
		\draw[thin, black] (q24) -- (q46) ;
		\node[xshift=2,yshift=8] at (q46) {$\mathbf{x}'$} ;
		\draw[thin,black, dashed] (q46) -- (r6) ;
		
		\draw[thick, black] (l,N,0)-- (l,0,0) -- (0,0,0);
		\draw[thick, black] (0,0,N) -- (0,N,N) -- (l,N,N) -- (l,0,N) -- (0,0,N);
		\draw[thick, black] (0,0,0) -- (0,0,N);
		\draw[thick, black] (l,0,0) -- (l,0,N);
		\draw[thick, black] (l,N,0) -- (l,N,N);

		\foreach \y in {1,2,3,4}
		{
			\draw[thin, gray] (\y*N,0,0) -- (\y*N,0,N);
			\draw[thin, gray] (\y*N,0,N) -- (\y*N,N,N);
		}
		\end{tikzpicture}
	\end{center}
	\caption{Illustration for Claim \ref{fatblocks}. Here, $\bullet$ is used for the elements of $\F$ whose projection is the same for all three blocks, and $\circ$ are the positions of the other elements. (The projections of the $\bullet$ points form $\mathcal{G}$.) First, we find a strong copy $S$ of $P$ in $\mathcal{G}$. Then, for each element $\mathbf{x}\in S$ we pick $\mathbf{x}'$ in the preimage of $\mathbf{x}$ from the block corresponding to the rank of $\mathbf{x}$ in the poset induced on $S$.
     } 
	\label{figure2}
	
\end{figure}

For $i\in [h]$, define the equivalence classes $\equiv_{i}$ similarly as in Lemma \ref{densitylemma}. For $x,y\in [k]$, let $x\equiv_{i} y$ if $\lceil x/s_{i}\rceil=\lceil y/s_{i}\rceil$. Also, let $x\rightarrow_{i}y$ if $x<y$ and $x\equiv_{i}y$, but $x\not\equiv_{i-1} y$. Moreover, if $\mathbf{x},\mathbf{y}\in[k]^{d}$, let $\mathbf{x}\equiv \mathbf{y}$ if $\mathbf{x}(d)\equiv_{i} \mathbf{y}(d)$, and let $\mathbf{x}\rightarrow_{i} \mathbf{y}$ if $\mathbf{x}(d)\rightarrow_{i}\mathbf{y}(d)$.

Say that an element $\mathbf{x}\in \F$ is \emph{good} if for $i\in[h]$, there exist $\mathbf{x}^{(-i)},\mathbf{x}^{(i)}$ such that $\mathbf{x}^{(-i)}\rightarrow_{i} \mathbf{x}$ and $\mathbf{x}\rightarrow_{i} \mathbf{x}^{(i)}$. Let $\mathcal{G}$ be the family of good elements.

\begin{claim}\label{goodclaim}
$|\mathcal{G}|\leq C_{l-1}k^{d-1}.$
\end{claim}

\begin{proof}
Suppose that $|\mathcal{G}|>C_{l-1}k^{d-1}$. Then there exists $t$ such that $$\mathcal{G}_{t}=\{(x_{1},\dots,x_{d-1}): (x_{1},\dots,x_{d-1},t)\in \mathcal{G}\}$$ has more than $C_{l-1}k^{d-2}$ elements. But then by our induction hypothesis, $\mathcal{G}_{t}$ contains a strong copy of $P_{l-1}$.
 In the exact same way as in the proof of Lemma \ref{densitylemma}, we can use this strong copy of  $P_{l-1}$ to build a strong copy of $P_{l}$ in $\F$, contradiction.
\end{proof}
 
Now let us bound the size of $\mathcal{G}$ by considering medium blocks. Let $\mathcal{H}_{i}$ be the set of elements $\mathbf{x}\in\F$ such that there exist no $\mathbf{x}^{(-i)}$ or $\mathbf{x}^{(i)}$.

\begin{claim}\label{badelements}
If $i\in [h-1]$, we have $|\mathcal{H}_{i}|\leq |\F^{i}_{1}|+|\F^{i}_{2}|+2|\F^{i}_{3}|/p.$

If $i=h$, $|\mathcal{H}_{h}|\leq 2s_{h-1}k^{d-1}$.
\end{claim}

\begin{proof}
First, suppose that $i\in[h-1]$. It is enough to show that if $\mathbf{u}\in [k/s_{i}]^{d}$ is medium, then $|B^{i}_{\mathbf{u}}\cap \mathcal{H}_{i}|\leq 2|B^{i}_{\mathbf{u}}\cap\F|/p$.

Let $(x_{1},\dots,x_{d-1})\in \pr( B^{i}_{\mathbf{u}}\cap \F)$ and let $t_{1}<\dots<t_{r}$ be the $d$-th coordinates of the points $\mathbf{y}_{1},\dots,\mathbf{y}_{r}\in B^{i}_{\mathbf{u}}\cap \F$, whose first $d-1$ coordinates are $(x_{1},\dots,x_{d-1})$. Clearly, $\mathbf{y}_{1}\equiv_{i}\dots\equiv_{i}\mathbf{y}_{r}$. However, $\mathbf{y}_{j}\not\equiv_{i-1} \mathbf{y}_{j+s_{i-1}}$ for $1\leq j<j+s_{i-1}\leq r$. Hence, if $s_{i-1}<j\leq r-s_{i-1}$, then $\mathbf{y}_{1}\rightarrow_{i} \mathbf{y}_{j}$ and $\mathbf{y}_{j}\rightarrow_{i} \mathbf{y}_{r}$, which means that $\mathcal{H}_{i}$ contains at most $2s_{i-1}$ points among $\mathbf{y}_{1},\dots,\mathbf{y}_{r}$. 

But this is true for every element of the projection $\pr( B^{i}_{\mathbf{u}}\cap \F)$, so $$|\mathcal{H}_{i}\cap B^{i}_{\mathbf{u}}|\leq 2s_{i-1}|\pr( B^{i}_{\mathbf{u}}\cap \F)|< 2s_{i}^{d-1}/p^{2}\leq 2|B^{i}_{\mathbf{u}}\cap\F|/p,$$
where the last two inequality holds because $\mathbf{u}$ is medium. 

Now let $h=i$. Just as before, each row of $[k]^{d}$ parallel to the $d$-th coordinate axis can contain at most $2s_{h-1}$ elements of $\mathcal{H}_{h}$. Hence, $|\mathcal{H}_{h}|\leq 2s_{h-1}k^{d-1}$.
\end{proof}

By using the trivial bound $|\F|\geq |\F_{3}^{i}|$, we also have 
$$|\mathcal{H}_{i}|\leq |\mathcal{F}_{1}^{i}|+|\mathcal{F}_{2}^{i}|+2|\F|/p$$ 
for $i\in [h-1]$.

Note that $\mathcal{G}=\F\setminus(\bigcup_{i=1}^{h} \mathcal{H}_{i})$, so Claim \ref{badelements} implies 
\begin{equation}\label{equ2}
|\mathcal{G}|\geq |\F|-\sum_{i=1}^{h} |\mathcal{H}_{i}|\geq (1-2h/p)|\F|-\sum_{i=1}^{h-1}|\mathcal{F}^{i}_{1}|-\sum_{i=1}^{h-1}|\mathcal{F}^{i}_{2}|-2s_{h-1}k^{d-1}.
\end{equation}

In what comes, we shall bound the three negative terms of the right hand side of this inequality using our earlier claims. First, we have the trivial bound that $2s_{h-1}k^{d-1}\leq C_{l}k^{d-1}/p$.

Now let us move on to the first sum on the right hand side of (\ref{equ2}). We have $|\F^{i}_{1}|\leq |\mathcal{I}^{i}_{1}|s_{i}^{d-1}/p$ as every light block contains at most $s_{i}^{d-1}/p$ elements. Combining this with Claim \ref{lightblocks}, we get 
$$\sum_{i=1}^{h-1}|\mathcal{F}_{1}^{i}|\leq \sum_{i=1}^{h-1} C_{l}(k/s_{i})^{d-1}s_{i}^{d-1}/p<hC_{l}k^{d-1}/p.$$

Finally, let us consider the second sum on the right hand side of (\ref{equ2}). As any block of size $s_{i}$ can contain at most $s_{i}^{d}$ elements, we have $|\F^{i}_{2}|\leq s_{i}^{d}|\mathcal{I}^{i}_{2}|$. Hence, Claim \ref{fatblocks} implies that
$$\sum_{i=1}^{h-1}|\mathcal{F}_{2}^{i}|\leq \sum_{i=1}^{h-1}2hp^{2}s_{i-1}s_{i}k^{d-1}\leq 2h^{2}p^{2}s_{h-1}^{2}k^{d-1}\leq C_{l}k^{d-1}/p.$$

Plugging these three upper bounds in (\ref{equ2}), we arrive to the inequality
$$|\mathcal{G}|\geq (1-2h/p)|\F|-(h+2)C_{l}k^{d-1}/p\geq (1-2h/p)|\F|-2hC_{l}k^{d-1}/p.$$
 But Claim \ref{goodclaim} also provides an upper bound on $|\mathcal{G}|$. Hence, we have
$$C_{l-1}k^{d-1}\geq (1-2h/p)|\F|-2hC_{l}k^{d-1}/p,$$
which yields
$$|\F|\leq (C_{l-1}+2hC_{l}/p)k^{d-1}/(1-2h/p)\leq C_{l}k^{d-1},$$
where the last inequality holds by the definition of $C_{l}$. This finishes the proof.

\end{proof}

Now we are in a position to easily conclude the proof of our main theorem, by combining Theorem \ref{mainthm2} and Theorem \ref{gridtogrid}.

\begin{proof}[Proof of Theorem \ref{mainthm}.]
	Let $d=2|P|$ and let $c(h)$ be a function given by Theorem \ref{mainthm2}. Then every (induced $P$)-free family $\F'\subset [k]^{d}$ satisfies $|\F'|\leq |P|^{c(h)}k^{d-1}$. But then by Theorem \ref{gridtogrid}, for every $n\geq 2|P|$ and every (induced $P$)-free $\F\subset [k]^{n}$, we have $|\F|\leq O(\sqrt{2|P|}|P|^{c(h)}w)$, where $w$ is the width of $[k]^{n}$. Setting $c_{h}=c(h)+O(1)$ finishes our proof.
\end{proof} 

Now let us turn to the special case of $h=2$, which is Theorem \ref{bipartitethm}. The proof of this theorem  is the same as the proof of Theorem \ref{mainthm}, but we choose our parameters more carefully. We shall only sketch the proof.

\begin{proof}[Proof of Theorem \ref{bipartitethm} (Sketch).]
	Let $P$ be a poset of height $2$. Let $A$ be the set of minimal elements of $P$, and let $B$ be the set of maximal elements of $P$. Without loss of generality, let $a=|A|$, $b=|B|$, $a\leq b$. Then $q=a$ and $d_{0}=O(\log b)$. Also, Lemma \ref{densitylemma} yields that for $l\in \{0,1,...,a\}$, if $\F\subset [k]^{d_{0}+l}$ does not contain a strong copy of $P_{l}$, then $|\F|=O(ak^{d_{0}+l-1/2}+k^{d_{0}+l-1}\log b)$.
	
	Now consider the proof of Lemma \ref{mainlemma} for this special case. One can choose the parameters $s_{1}=a^{\Theta(1)}(\log b)^{\Theta(1)}$, $C_{0}=a^{\Theta(1)}(\log b)^{\Theta(1)}$ and $C_{i+1}=C_{i}(1+O(1/a))$ to show that if the family $\F\subset [k]^{d_{0}+l+1}$ does not contain a strong copy of $P_{l}$, then $|\F|\leq C_{l}k^{d_{0}+l}$. But then $C_{q}=C_{a}=a^{\Theta(1)}(\log b)^{\Theta(1)}$, so if the family $\F\subset [k]^{d_{0}+a+1}$ does not contain a strong copy of $P_{q}=P_{a}=P$, then $|\F|=a^{\Theta(1)}(\log b)^{\Theta(1)}k^{d_{0}+a}$.
	
	But then by Theorem \ref{gridtogrid}, for every $n\geq=d_{0}+a=a+\Theta(\log b)$ and every (induced $P$)-free $\F\subset [k]^{n}$, we have 
	$$|\F|= a^{\Theta(1)}(\log b)^{\Theta(1)}w,$$ where $w$ is the width of $[k]^{n}$.
	
\end{proof}

\section{Acknowledgments}

I would like to thank D\'{a}niel Kor\'{a}ndi for the simple proof of Lemma \ref{intersectionlemma}.


\begin{thebibliography}{99}
	
\bibitem{A}
I. Anderson,
\emph{Combinatorics of Finite Sets,}
Oxford University Press, 1987.

\bibitem{BJ}
E. Boehnlein, T. Jiang,
\emph{Set families with a forbidden induced subposet,}
Combinatorics, Probability and Computing 21 (2012): 496-511.

\bibitem{B}
B. Bukh, 
\emph{Set families with a forbidden subposet,}
 Electronic J. of Combinatorics, 16 (2009): R142, 11p.

\bibitem{BN}
P. Burcsi, D. T. Nagy, 
\emph{The method of double chains for largest families with excluded subposet,}
Electronic Journal of Graph Theory and Applications 1 (2013): 40-49.

\bibitem{CK}
T. Carroll, G. O. H. Katona,
\emph{Bounds on maximal families of sets not containing three sets with $A\cap B\subset C$, $A\not\subset B$,}
Order 25 (3) (2008): 229-236.

\bibitem{CL}
H. B. Chen, W. T. Li,
\emph{A Note on the Largest Size of Families of Sets with a Forbidden Poset,} 
Order 31 (2014): 137-142.

\bibitem{DK} 
A. De Bonis, G. O. H. Katona, 
\emph{Largest families without an r-fork,}
Order 24 (2007): 181-191.

\bibitem{DKS}
 A. De Bonis, G. O. H. Katona and K. J. Swanepoel,
 \emph{Largest family without $A\cup B\subset C\cap D$,}
J. Combin. Theory. (Ser. A) 111 (2005): 331-336.

\bibitem{E}
P. Erd\H{o}s,
\emph{On a lemma of Littlewood and Offord.}
Bulletin of the American Mathematical Society, 51 (12) (1945): 898-902.

\bibitem{E2}
P. Erd\H{o}s, 
\emph{On extremal problems of graphs and generalized graphs,}
Israel J. Math. 2 (1964): 183-190.

\bibitem{EK}
P. Erd\H{o}s, D. J. Kleitman,
\emph{ On collections of subsets containing no 4-member Boolean
algebra,} Proc. Amer. Math. Soc. 28 (1971): 87-90.

\bibitem{Fox}
J. Fox,
\emph{Stanley-Wilf limits are typically exponential,}
arXiv:1310.8378

\bibitem{GT}
J. T. Geneson, P. M. Tian,
\emph{Extremal Functions of Forbidden Multidimensional Matrices,}
arXiv:1506.03874

\bibitem{GLL}
J. R. Griggs, W.T. Li, L. Lu, 
\emph{Diamond-free families,}
 J. Combin. Theory. (Ser. A) 119 (2012): 310-322.
 
 \bibitem{GL}
 \emph{J. R. Griggs and L. Lu, On families of subsets with a forbidden subposet,}
 Combinatorics, Probability, and Computing 18 (2009), 731-748.

\bibitem{GRS}
D. S. Gunderson, V. R\"{o}dl, A. Sidorenko, 
\emph{Extremal problems for sets forming Boolean algebras and complete partite hypergraphs,}
Journal of Combinatorial Theory A, 88 (2) (1999): 342-367.


\bibitem{GMT}
D. Gr\'{o}sz, A. Methuku, C. Tompkins,
\emph{An improvement of the general bound on the largest family of subsets avoiding a subposet,} 
Order 34 (1) (2017):113-125.


\bibitem{JLM}
T. Johnston, L. Lu, K. G. Milans,
\emph{Boolean algebras and Lubell functions,}
Journal of Combinatorial Theory A 136 (2015), 174-183.

\bibitem{KT}
G. O. H. Katona, T. G. Tarj\'{a}n,
\emph{Extremal problems with excluded subgraphs in the n-cube,in Graph Theory}
 (M. Borowiecki, J.W. Kennedy and M. M. Sys\l{}o eds.), Springer, Berlin, (1983): 84-93.

\bibitem{KM}	
M. Klazar, A. Marcus, 
\emph{Extensions of the linear bound in the F\"{u}redi-Hajnal conjecture,}
 Adv. in Appl. Math. 38 (2006): 258-266.

\bibitem{K2}
D. J. Kleitman,
\emph{Extremal Properties of Collections of Subsets Containing No Two Sets and Their Union,}
Journal of Combinatorial Theory A 20 (1976): 390-392.

\bibitem{KMY}
L. Kramer, R. R. Martin, M. Young, 
\emph{On diamond-free subposets of the Boolean lattice,}
J. Combin. Theory. (Ser. A) 120 (2013): 545-560.

\bibitem{LM}
L. Lu, K. G. Milans,
\emph{Set families with forbidden subposets,}
Journal of Combinatorial Theory Ser. A, 136 Issue C (2015): 126-142. 

\bibitem{MT}	
A. Marcus, G. Tardos,
\emph{Excluded permutation matrices and the Stanley-Wilf conjecture,}
J. Combin. Theory. (Ser. A) 107 (2004): 153-160.	

\bibitem{M}
A. M\'{e}roueh,
\emph{Lubell mass and induced partially ordered sets,}
arXiv:1506.07056
	
\bibitem{MP}
A. Methuku, D. P\'{a}lv\"{o}lgyi,
\emph{Forbidden Hypermatrices Imply General Bounds on Induced Forbidden Subposet Problems,}
Combinatorics, Probability and Computing,  doi:10.1017/S0963548317000013.

\bibitem{S}
E. Sperner,
\emph{"Ein Satz \"{u}ber Untermengen einer endlichen Menge"},
Mathematische Zeitschrift (in German), 27 (1) (1928): 544-548.

\bibitem{TH}
H. T. Thanh, 
\emph{An extremal problem with excluded subposet in the Boolean lattice,}
Order 15 (1998): 51-57.


\bibitem{T}
I. Tomon,
\emph{On a conjecture of F\"{u}redi,}
 European Journal of Combinatorics 49 (2015): 1-12.
 
\bibitem{T2}
I. Tomon,
\emph{Forbidden induced subposets in the grid,}
arXiv:1705.09551


\end{thebibliography}
\end{document}